\documentclass[10pt]{article}
\usepackage[utf8]{inputenc} 
\usepackage{amsmath, amssymb, amsthm,graphicx}
\usepackage{natbib}
\newtheorem{lemma}{Lemma}

\newtheorem{algorithm}{Algorithm}
\newtheorem{prop}{Proposition}
\newtheorem{thm}{Theorem}

\def \n{\Vert}

\newcommand{\E}{\mathop{\mathbb{E}}}
\newcommand{\dist}{\mathop{\mathrm{dist}}}
\renewcommand{\P}{\mathop{\mathbb{P}}}

\def\R{{\mathbb{R}}}
\def\N{{\mathbb{N}}}

\def\|{\,|\,}
\def\bn#1\en{\begin{align*}#1\end{align*}}
\def\bnn#1\enn{\begin{align}#1\end{align}}

\title{Optimal worst-risk minimization in structural equation models with random coefficients}
\date{}
\author{Philip Kennerberg and Ernst C. Wit}
\begin{document} 
\maketitle

\begin{abstract}
The insight that causal parameters are particularly suitable for out-of-sample prediction has sparked a lot development of causal-like predictors. However, the connection with strict causal targets, has limited the development with good risk minimization properties, but without a direct causal interpretation. In this manuscript we derive the optimal out-of-sample risk minimizing predictor of a certain target $Y$ in a non-linear system $(X,Y)$ that has been trained in several within-sample environments. We consider data from an observation environment, and several shifted environments. Each environment corresponds to a structural equation model (SEM), with random coefficients and with its own shift and noise vector, both in $L^2$. Unlike previous approaches, we also allow shifts in the target value. We define a sieve of out-of-sample environments, consisting of all shifts $\tilde{A}$ that are at most $\gamma$ times as strong as any weighted average of the observed shift vectors. For each $\beta\in\R^p$ we show that the supremum of the risk functions $R_{\tilde{A}}(\beta)$ has a worst-risk decomposition into a (positive) non-linear combination of risk functions, depending on $\gamma$. We then define the set $\mathcal{B}_\gamma$, as minimizers of this risk. The main result of the paper is that there is a unique minimizer ($|\mathcal{B}_\gamma|=1$) that can be consistently estimated by an explicit estimator, outside a set of zero Lebesgue measure in the parameter space. A practical obstacle for the initial method of estimation is that it involves the solution of a general degree polynomials. Therefore, we prove that an approximate estimator using the bisection method is also consistent.
\end{abstract}

\section{Introduction}

Minimization of prediction error is a standard way for evaluating models. It is used in cross-validation techniques \citep{laan2006cross, van2007super}, in information criteria, such as the AIC \citep{akaike1974new} or Mallow's Cp \citep{gilmour1996interpretation}. All of these methods assume that the circumstances between the observed training data and future test data do not change, i.e., the underlying distribution is entirely reflected in the sampling of the training data. This is often not the case. Either the sampling may consist of a particular subpopulation of the full population of interest, or there may be a temporal separation between the training of the model and its application that may have been punctuated by external shocks, as for example is a typical scenario in economics and finance \citep{kremer2018risk}.

In the field of statistics, risk minimization has received renewed attention in the past decade. 
\cite{icp} introduced invariant causal prediction as a way for predicting the outcome of a causal relationship between variables while maintaining its validity across different environments or contexts. It involves identifying causal relationships that are robust and invariant, meaning they hold true even when the conditions or settings change. \cite{causaldantzig} proposed the Causal Dantzig as a method for characterizing causal effects in observational studies by formulating it as a linear optimization problem. The empirical version of the Dantzig selector used idea of sparsity, assuming that only a few covariates have a direct causal impact on the outcome variable. By minimizing a penalized loss function subject to a set of linear constraints, Causal Dantzig selects a subset of relevant covariates and estimates their causal effects while controlling for confounding variables. 

\cite{Rot} proposed anchor regression as a statistical method for estimating the causal effect of a treatment or intervention on an outcome variable by accounting for potential confounding variables. It involves selecting a subset of covariates, known as \emph{anchor variables}, that are strongly associated with both the treatment and outcome variables. By regressing the outcome variable on the treatment and anchor variables, anchor regression helps mitigate bias caused by confounding factors and provides a more accurate estimation of the treatment effect. It serves as a useful tool in causal inference and allows researchers to make informed decisions about the effectiveness of interventions while controlling for potential confounding variables. \cite{kania2022causal} introduced a related approach, called \emph{causal regularization}. This method does not require explicit information on the auxilary variables, and which comes with strong out-of-sample risk guarantees. \cite{shen2023causality} propose \emph{distributional robustness via invariant gradients} (DRIG), a method that exploits general additive interventions in training data for robust predictions against unseen interventions. They establish robustness guarantees of DRIG under a linear structural causal model.

One limitation of previously mentioned methods is that they typically assume very restrictive settings, such as a linear SEM structure, exactly two observational environments or no intervention on the target variable. The aim of this manuscript is to relax these assumptions. In section~\ref{sec:outofsamplerisk} we introduce the multi-environment setting, where each environment corresponds to a non-linear structural equation model (SEM) with its own $L^2$ additive shift and $L^2$ additive noise vector. In section~\ref{sec:optimalCR}, we define the space of all out-of-sample environments $C^\gamma$, that correspond to shifts that are at most $\gamma$ times as strong as any weighted average of the observed shift vectors. For each $\beta\in\R^p$ we show that if we consider the supremum of the risk functions $R_{\tilde{A}}(\beta)$ over $\tilde{A}\in C^\gamma$ then this supremum has a worst-risk decomposition into a (positive) linear combination of risk functions, depending on $\gamma$. Analogous to \cite{kania2022causal}, we then define the risk minimizer, $\mathcal{B}_\gamma$, as the set of arguments $\beta$ that minimizes this worst risk. In section~\ref{sec:estimation} we show that there is a unique minimizer ($|\mathcal{B}_\gamma|=1$) and that this minimizer can be consistently estimated with an explicit estimator (that by-passes the issue of having an unbounded search-space) outside a set of zero Lebesgue measure in the parameter space. A practical obstacle for such estimation is that it will involve solution of a general degree polynomial. Therefore we also prove that an approximate plug-in estimator using the bisection method is also consistent. An interesting by-product of the proof is that plug-in estimation of the argmin of the maxima of a finite set of quadratic risk functions is consistent outside a set of zero Lebesgue measure in the parameter space. 


\section{Structural equation model with random coefficients}
\label{sec:outofsamplerisk}
Assume we are given a probability space $\left(\Omega,\mathcal{F},\P \right)$. For $1\le i\le k$, $k\in\N$, let $Y^{A_i}\in\R$, $X^{A_i}\in\R^p$ be random variables and vectors respectively on this space that are solutions to the following structural equations, 
\begin{equation}\label{SEMA}
	\begin{bmatrix}
		Y^{A_i}\\
		X^{A_i}
	\end{bmatrix} = 
	B(\omega)\cdot\begin{bmatrix}
		
		Y^{A_i}\\
		X^{A_i}
	\end{bmatrix}
	+\epsilon_{A_i}
	+
	A_i
\end{equation}
where $B(\omega)$ is a random real-valued $(p+1)\times(p+1)$ matrix such that $I-B$ is full rank a.s., the components of $A_i\in\R^{p+1}$, $\epsilon_{A_i}\in\R^{p+1}$ are in $L^2(\P)$ for $1\le i\le k$. Note that $B$ is the same random matrix for all equation systems. We will refer to these $k$ equation systems as environments. Stochastically, the roles of $X^{.}$ and $Y^{.}$ are completely identical, but our prediction focus is on the \emph{target} $Y^.$. The random vector $A_i\in \R^{p+1}$ is called the shift corresponding to environment $i$. Since $I-B$ has full rank a.s., $X^{A_i}$ and $Y^{A_i}$ have unique solutions. We also consider the observational (shift free) environment 
\begin{equation}\label{SEMO}
	\begin{bmatrix}
		Y^{O}\\
		X^{O}
	\end{bmatrix} = 
	B(\omega)\cdot\begin{bmatrix}
		Y^{O}\\
		X^{O}
	\end{bmatrix}
	+\epsilon_{O}.
\end{equation}
We assume that $\epsilon_O$ and $B$ have the same joint law as $\epsilon_{A_i}$ and $B$ for all $1\le i\le k$. Denote $\sigma(B)$ as the sigma algebra generated by the entries of $B$ and let $\E\left[.\|B\right]$ denote conditional expectation with respect to $\sigma(B)$. We also assume that the noise terms are uncorrelated with the shifts given $B$ i.e. $\E\left[\epsilon_{A_i}A_i^T\| B\right]=0$ a.s..

\subsection{Non-linear interpretation of the random transfer matrix}
A few words about what the implications are of having a random transfer matrix $B$ are in order here. One obvious benefit is that since $B$ is random, this allows for extra randomness not captured by some corresponding linear model. It will allow for hidden confounding that enters multiplicatively in the same manner across all environments (including the observational). It may also help to introduce certain kinds of non-linearity. We can simultaneously fit any $k\in\N$ number of non-linear environments to a system of the type \eqref{SEMA}-\eqref{SEMO} (albeit with different shifts and noise) on the same probability space. Consider first $k\le p+1$ non-linear systems of the form
\begin{equation}\label{SSA}
	\begin{bmatrix}
		Y^{\tilde{A}_i}\\
		X^{\tilde{A}_i}
	\end{bmatrix} = 
	f\left(\begin{bmatrix}
		Y^{\tilde{A}_i}\\
		X^{\tilde{A}_i}
	\end{bmatrix} + \tilde{A}_i\right)
	+\eta_{\tilde{A}_i},
\end{equation}
for $0\le i\le p$ ($i=0$ will correspond to the observational environment where $\tilde{A}_0=0$), where $f:\mathbb{R}^{p+1}\rightarrow\mathbb{R}^{p+1}$ is a measurable function and $A_i,\eta_{A_i}\in\R^{p+1}$ are random vectors. We are now tasked with finding $\epsilon_O,\epsilon_{A_1},...,\epsilon_{A_p},A_1,...,A_p$ and $B$ such that the conditional dependence assumptions of the systems in \eqref{SEMA}-\eqref{SEMO} are met Let $\epsilon_O=\epsilon_{A_1}=...=\epsilon_{A_p}:=\epsilon=(1,0,...,0)$ and $A_i(l)=\delta_{i,l+1}$. Denote the $(p+1)\times(p+1)$ matrices
\begin{eqnarray}
	C&=&\begin{bmatrix}
		\epsilon;\epsilon+A_1;...;\epsilon+A_p
	\end{bmatrix} , \nonumber 
	\\
	D(\omega)&=& \begin{bmatrix}
		f\left(\begin{bmatrix}
			Y^{\tilde{O}}\\
			X^{\tilde{O}}
		\end{bmatrix}\right)
		+\eta_{\tilde{A}_0};
		f\left(\begin{bmatrix}
			Y^{\tilde{A}_1}\\
			X^{\tilde{A}_1}
		\end{bmatrix} + \tilde{A}_1\right)
		+\eta_{\tilde{A}_1};...;
		f\left(\begin{bmatrix}
			Y^{\tilde{A}_p}\\
			X^{\tilde{A}_p}
		\end{bmatrix} + \tilde{A}_p\right)
		+\eta_{\tilde{A}_p}\end{bmatrix}. \label{CD}
\end{eqnarray}
We can fit $B$ to the $p+1$  non-linear environments if we can solve the matrix equation $(I-B)^{-1}C =D$, which is equivalent (if $D$ is full rank) $B=I-D^{-1}C$ while still having $I-B$ being full rank (a.s.). This is possible if and only if $D$ and $C$ are full rank a.s., and $C$ was already chosen to have full rank. So with $D$ being full rank we can thus find a $B(\omega)$ for the systems \eqref{SEMA}-\eqref{SEMO} while also solving the corresponding non-linear systems in \eqref{SSA}, i.e.,
\begin{equation}\label{kpplus1}
	\begin{bmatrix}
		Y^{\tilde{O}}\\
		X^{\tilde{O}}\\
		Y^{\tilde{A}_1}\\
		X^{\tilde{A}_1}\\
		\vdots\\
		Y^{\tilde{A}_k}\\
		X^{\tilde{A}_k}
	\end{bmatrix} =
	\begin{bmatrix} 
		f\left(\begin{bmatrix}
			Y^{\tilde{O}}\\
			X^{\tilde{O}}
		\end{bmatrix} \right)\\
		f\left(\begin{bmatrix}
			Y^{\tilde{A}_1}\\
			X^{\tilde{A}_1}
		\end{bmatrix}+\tilde{A}_1 \right)\\
		\vdots\\
		f\left(\begin{bmatrix}
			Y^{\tilde{A}_k}\\
			X^{\tilde{A}_k}
		\end{bmatrix}+\tilde{A}_k \right)
	\end{bmatrix}
	+\begin{bmatrix}
		\eta_{\tilde{O}}\\
		\eta_{\tilde{A}_1}\\
		\vdots\\
		\eta_{\tilde{A}_k}
	\end{bmatrix}
	=
	\left((I-B)^{-1}
	\begin{bmatrix}
		\epsilon;
		\epsilon+A_1;
		...
		\epsilon+A_k;
	\end{bmatrix}\right)^T
	=
	\begin{bmatrix}
		Y^{O}\\
		X^{O}\\
		Y^{A_1}\\
		X^{A_1}\\
		\vdots\\
		Y^{A_p}\\
		X^{A_p}
	\end{bmatrix}
\end{equation}
Let us now deal with the case when $k>p+1$. First extend $f$ to $\tilde{f}:\R^k\to\R^k$, as $\tilde{f}=(f,0..,0)$. Similarly extend $X^{\tilde{A_i}}$ to $X'^{\tilde{A_i}}\in\R^{k-1}$, with  $X'^{\tilde{A_i}}(l)=\eta_{\tilde{A_i}}(l)$ for $l> p+1$ and for $l\le  p+1$, $X'^{\tilde{A_i}}(l)=X^{\tilde{A_i}}(l)$. We extend the shifts, $\tilde{A_i}$ to $\tilde{A_i}'\in\R^{k}$, with $\tilde{A_i}'(l)=\tilde{A_i}(l)$ for $l\le p+1$ and $\tilde{A_i}'(l)=0$ for $p+1<l\le k$. Finally we extend the noise, to $\eta_{A_i}'\in\R^k$ by $\eta_{A_i}'(l)=\eta_{A_i}(l)$ for $1\le l\le p+1$ and $\eta_{A_i}'(l)=Z_{i,l}$ for $p+1<l\le k$, where $\{Z_{i,l}\}_{i,l}$ is some set of absolutely continuous random variables that are independent of the rest of the system and amongst each other. In our new construction we have $p'+1=k$ (where $p'$ is the number of covariates in the extended system) so we can now apply to former construction in \eqref{CD} to get the corresponding solution in \eqref{kpplus1}. The first $p+1$ rows in
\begin{equation*}\label{SSAprime}
	\begin{bmatrix}
		Y^{\tilde{A}'_i}\\
		X'^{\tilde{A}'_i}
	\end{bmatrix} = 
	\tilde{f}\left(\begin{bmatrix}
		Y^{\tilde{A}'_i}\\
		X'^{\tilde{A}'_i}
	\end{bmatrix} + \tilde{A}'_i\right)
	+\eta'_{\tilde{A}'_i}
	=
	\begin{bmatrix}
		Y^{A_i}\\
		X^{A_i}
	\end{bmatrix},
\end{equation*}
are exactly the same as those in \eqref{SSA}, which means that our desired system of the form \eqref{SEMA}-\eqref{SEMO} is given by
\begin{equation*}
	\begin{bmatrix}
		Y^{A_i}\\
		X^{A_i}(1:p)
	\end{bmatrix},
\end{equation*}
where
\begin{equation*}
	\begin{bmatrix}
		Y^{A_i}\\
		X^{A_i}
	\end{bmatrix}=(I-B')^{-1}C',
\end{equation*}
$B'=I-D'^{-1}C$, $C$ is defined as before and 
\begin{equation*}
	D'(\omega)= \begin{bmatrix}
		\tilde{f}\left(\begin{bmatrix}
			Y'^{\tilde{O}}\\
			X'^{\tilde{O}}
		\end{bmatrix}\right)
		+\eta'_{\tilde{A}_0};
		\tilde{f}\left(\begin{bmatrix}
			Y'^{\tilde{A}_1}\\
			X'^{\tilde{A}_1}
		\end{bmatrix} + \tilde{A}'_1\right)
		+\eta_{\tilde{A}_1};...;
		\tilde{f}\left(\begin{bmatrix}
			Y'^{\tilde{A}_p}\\
			X'^{\tilde{A}_p}
		\end{bmatrix} + \tilde{A}'_p\right)
		+\eta'_{\tilde{A}_p}\end{bmatrix}.
\end{equation*}
%
%

\subsection{Out-of-sample environments}
We assume that the $k+1$ environments defined above will constitute the observed states of the system. Additionally, in this section we define a sieve of out-of-sample extensions that constitute potential future observations of the same system, in which we aim to minimize the worst possible prediction risk. We start by defining an arbitrary out-of-sample direction $w$ in which the future shift $A_w$ may occur. After that, we define the space of out-of-sample environments $\mathcal{C}_w^\gamma$ in the $w$ direction that are at most $\gamma$ strong.  

Let $\n.\n$ be the Euclidean norm and define the set of weights, $\mathcal{W}=\left\{w\in\R^k: \n w\n=1\right\}$ (to be clear, these weights are deterministic and moreover this choice of $\mathcal{W}$ is in some sense arbitrary, any compact set in $\R^k$ works). For a vector $w\in\mathcal{W}$ we define $A_w=\sum_{i=1}^kw_iA_i$. Given a joint distribution $F_A$ on $\R^{p+1}$ whose marginals have finite second moments, let $(\Omega_A,\mathcal{F}_A,\P_A)$ denote an extension of the original probability space that also supports a random vector $A\in\R^{p+1}$ distributed according to $F_A$, independent from $B$, and a random vector $\epsilon_A$ with the same distribution as $\epsilon_O$ and that is also independent from both $A$ and $B$. On this space we may define $Y^{A}\in\R$ and $X^{A}\in\R^p$ through,
\begin{equation}\label{A}
	\begin{bmatrix}
		Y^{A}\\
		X^{A}
	\end{bmatrix} := 
	(I-B)^{-1}(\epsilon_{A}
	+
	A),
\end{equation}
which is the solution to the structural equation system,
\begin{equation*}
	\begin{bmatrix}
		Y^{A}\\
		X^{A}
	\end{bmatrix} = 
	B(\omega)\cdot\begin{bmatrix}
		
		Y^{A}\\
		X^{A}
	\end{bmatrix}
	+\epsilon_{A}
	+
	A.
\end{equation*}

Let $\mathcal{P}$ denote some set of distributions on $\R^{p+1}$ whose marginals have finite second moments and such that $F_{A_w}\in\mathcal{P}$ for any $w\in\mathcal{W}$. We define the shiftspace of distributions,
\begin{align}\label{shiftspace}
	C^\gamma_w(B)=\left\{F_A\in \mathcal{P}:\E_A\left[AA^T\|B\right]\preccurlyeq \gamma \E_A\left[A_wA_w^T\|B\right], \P_A\textsf{- a.s. }\right\}.
\end{align} 
The following proposition gives some important examples of how the shift-space is affected by certain properties of $B$.
\begin{prop} Special cases of out-of-sample environments. 
	\begin{itemize}
		\item[1)] if $B$ is any non-zero deterministic matrix then \eqref{SEMA} and \eqref{SEMO} become linear SEMs and \eqref{shiftspace} reduces to
		$$C^\gamma_w(B)=\left\{F_A\in \mathcal{P}:\E_A\left[AA^T\right]\preccurlyeq \gamma \E\left[A_wA_w^T\right]\right\}.$$ 
		\item[2)] If $B$ is a simple map taking (deterministic) values $\{B_1,...,B_m\}$ then \eqref{SEMA} and \eqref{SEMO} become piecewise linear SEMs, while \eqref{shiftspace} reduces to 
		$$C^\gamma_w(B)=\left\{F_A\in \mathcal{P}:\E_A\left[AA^T1_{B=B_l}\right]\preccurlyeq \gamma \E\left[A_wA_w^T1_{B=B_l}\right], 1\le l\le m\right\}.$$
		\item[3)] If $B$ is independent of $A_1,...,A_k$ then \eqref{shiftspace} reduces to 
		$$C^\gamma_w(B)=\left\{F_A\in \mathcal{P}:\sigma(A)\perp\sigma(B),\E_A\left[AA^T\right]\preccurlyeq \gamma \E\left[A_wA_w^T\right]\right\}.$$ 
		and the condition $\E_A\left[\epsilon A\|B\right]=0$ is equivalent to $\epsilon\E_A\left[A\right]=0$ a.s.. This means that all shifts have zero mean unless there is no noise a.s..
	\end{itemize}
\end{prop}
\begin{proof}
	Since 1) is a special case of both 2) and 3) it suffices to show these last two statements. For 2), we first note that $\E_A\left[AA^T\|B\right]\preccurlyeq \gamma \E_A\left[A_wA_w^T\|B\right]$ a.s. is equivalent to $\E_A\left[AA^T1_D\right]\preccurlyeq \gamma \E\left[A_wA_w^T1_D\right]$ (the subscript of $A$ can be dropped in the last expectation since $A_w$ is measurable with respect to $\mathcal{F}$) for all $D\in\sigma(B)$. Let $C_l=\{B=B_l\}$, for $1\le l\le m$. Since $\sigma(B)=\{\emptyset,C_1,...,C_m\}$, this inequality needs only to be verified for these sets, the inequality restricted to the empty set is tautological and is therefore not included in the condition. As for 3), the first statement follows directly from the independence of the shifts from $B$. For the second statement, note that if we define $f$ by 
	$$\epsilon_A A=\sum_{l=1}^{p+1}\epsilon_A(l)A(l)=f(\epsilon_A(1),...,\epsilon_A(p+1),A(1),...,A(p+1))$$ 
	then if we let 
	$$h(a_1,...,a_{p+1})=\E_A\left[f(a_1,...,a_{p+1},A(1),...,A(p+1))\right]=\sum_{l=1}^{p+1}a_l\E_A[A(l)]$$ 
	it follows that (see for instance section 9.10 in \cite{williams1991probability})
	\begin{align*}
		\E_A\left[\epsilon_A A\|B\right] &= \E_A\left[f(\epsilon_A(1),...,\epsilon_A(p+1),A(1),...,A(p+1))\|B\right]
		\\
		&=h(\epsilon_A(1),...,\epsilon_A(p+1))
		=\sum_{l=1}^{p+1}\epsilon_A(l)\E_A[A(l)]
		=\epsilon_A\E_A[A] \hspace{2mm} a.s.
	\end{align*}
\end{proof}

\section{The optimal worst-risk minimizer}
\label{sec:optimalCR}

We define the out-of-sample set of shifts (or rather, corresponding distributions) to that are at most $\gamma$ times as strong in any direction $w \in \mathcal{W}$  of the observed shifts, as follows
$$C^\gamma(B)=\left\{F_A'\in\mathcal{P}:\exists w\in\mathcal{W}, \E_A\left[A'(A')^T\|B\right]\preccurlyeq \gamma \E_A\left[A_w(A_w)^T\|B\right] a.s.\right\},$$
where, as before $\mathcal{P}$ denotes some set of distributions on $\R^{p+1}$ whose marginals have finite second moments and such that $F_{A_w}\in\mathcal{P}$ for any $w\in\mathcal{W}$. It is worth noting that
$$C^\gamma(B)=\bigcup_{w\in\mathcal{W}}C_w^\gamma(B).$$
From now on we will remove the dependence on $B$ from $C^\gamma(B)$ and $C_w^\gamma(B)$, simply writing $C^\gamma$ and $C_w^\gamma$ respectively instead. We now introduce a bit of notation. For any $F_A\in\mathcal{P}$, denote $R_{A}(\beta)=\E_A\left[\left(Y^A-\beta X^{A}\right)^2\right]$ and $R_{A_i}(\beta)=\E\left[\left(Y^{A_i}-\beta X^{A_i}\right)^2\right]$ for $0\le i\le k$ (with $R_O=R_{A_0}$). Let
\begin{itemize}
	\item[] $R^w_+(\beta)=\sum_{i=1}^kw_i^2R_{A_i}+R_O(\beta)$ and
	\item[] $R^w_\Delta(\beta)=\sum_{i=1}^kw_i^2R_{A_i}-R_O(\beta).$
\end{itemize}
We also set
\begin{itemize}
	\item[] $G_+=\E\left[(X^O)^TX^O+\sum_{i=1}^kw_i^2(X^{A_i})^TX^{A_i}\right]$,
	\item[]$G_\Delta=\E\left[\sum_{i=1}^kw_i^2(X^{A_i})^TX^{A_i}-(X^O)^TX^O\right]$,
	\item[] $Z_+=\E\left[(X^O)^TY^O+\sum_{i=1}^kw_i^2(X^{A_i})^TY^{A_i}\right]$ and $Z_\Delta=\E\left[\sum_{i=1}^kw_i^2(X^{A_i})^TY^{A_i}-(X^O)^TY^O\right].$
\end{itemize}
With the above definitions we may now present the "optimal" worst-risk decomposition (optimal in the sense of the chosen weights, minimizing the risk).
\begin{prop}\label{optimalsol}
	Let $\tau\ge -\frac 12$ and assume $Y^{A_0},...,Y^{A_p}\in L^2(\P)$ and that the components of $X^{A_0},...,X^{A_p}\in L^2(\P)$, then 
	\begin{align}\label{supw}
		\sup_{F_A\in C^{1+\tau}}R_{A}(\beta)=\frac{1}{2}R^{w^*}_+(\beta)+\frac{1+2\tau}{2}R^{w^*}_\Delta(\beta),
	\end{align}
	for some $w^*\in\mathcal{W}$. Moreover $w^*$ is achieved by setting $w_i=0$ if $R_{A_i}(\beta)<\max_{1\le l\le k}R_{A_l}(\beta)$ and then distributing $w^*$ among the environments such that $R_{A_i}(\beta)=\max_{1\le l\le k}R_{A_l}(\beta)$ (while still keeping $w\in\mathcal{W}$).
\end{prop}
\begin{proof}
Due to \eqref{A}, $Y^{A}=((I-B)^{-1})_{1,\textbf{.}}(A+\epsilon_A)$ and $X^{A}=((I-B)^{-1})_{2:p+1,\textbf{.}}(A+\epsilon_A)$. Since the entries of $(I-B)^{-1}$ are $\sigma(B)$-measurable it follows that if we define $v=\beta (I-B)^{-1}_{2:p+1,\textbf{.}}-(I-B)^{-1}_{1,\textbf{.}}$ (implying $v(\epsilon_A+A)=Y^A-\beta X^A$) then $v$ is also $\sigma(B)$-measurable. With this notation,
	\begin{align}\label{supeq1}
		\sup_{F_A\in C^{1+\tau}}R_{A}(\beta)
		&=
		\sup_{F_A\in C^{1+\tau}}\E_A\left[v\left(\epsilon_{A}+A\right)\left(\epsilon_{A}+A\right)^Tv^T\right]
		\nonumber
		\\
		&=
		\sup_{F_A\in C^{1+\tau}}\left(
		\E_A\left[v\epsilon_{A}\epsilon_{A}^Tv^T\right]
		+
		2\E_A\left[\E_A\left[v\epsilon_{A}A^Tv^T\|B\right]\right] 
		+
		\E_A\left[vAA^Tv^T\right]\right)
		\nonumber
		\\
		&=
		\sup_{F_A\in C^{1+\tau}}\left(
		\E_A\left[v\epsilon_{O}\epsilon_{O}^Tv^T\right]
		+
		2\E_A\left[v\E_A\left[\epsilon_{A}A^T\|B\right]v^T\right]
		+
		\sup_{F_A\in C^{1+\tau}}\E_A\left[vAA^Tv^T\right]\right)
		\nonumber
		\\
		&=\E\left[v\epsilon_{O}\epsilon_{O}^Tv^T\right] 
		+
		\sup_{F_A\in C^{1+\tau}}\E_A\left[vAA^Tv^T\right]
	\end{align}
	where we utilized the conditional independence of the noise and shift given $B$ and the fact that $v\epsilon_{A}A^Tv^T$ is a function of $B$ and $\epsilon_A$, which has the same joint law as $B$ and $\epsilon_O$. This implies that the supremum is attained, i.e., a maximum, if and only if $\sup_{F_A\in C^{1+\tau}}\E_A\left[vAA^Tv^T\right]$ attains a maximum. We now show that this is indeed the case. Note that since $C^\gamma=\bigcup_{w\in\mathcal{W}}C^\gamma_w$, any element in $C^\gamma$ can be identified with a $w\in\mathcal{W}$. Consider $L(A)=\E_A\left[AA^T\right]$ for $F_A\in C^{\gamma}$ and let $S=\sup_{F_A\in C^{\gamma}}L(A)$. There exists a sequence $\{F_{A_n}\}_n\subset C^\gamma$ such that $\lim_nL(A_n)=S$ and for each $n$, $F_{A_n}\in C^\gamma_{w_n}$ for some $w_n\in\mathcal{W}$. Since $\mathcal{W}$ is compact there exists a subsequence $\{w_{n_k}\}_k$ such that $w_{n_k}\to w^*$ for some $w^*\in\mathcal{W}$. Since $\lim_kL(A_{n_k})=S$ and $L(A_{n_k})\le \gamma L(A_{w_{n_k}})$ and $\lim_kL(A_{w_{n_k}})=L(A_{w^*})$ we must have that $S=L(A_{w^*})$, so the supremum is in fact a maximum attained at $A_{w^*}$. We now finish the proof of \eqref{supw}.
	For any $F_A\in C^{1+\tau}$ and $x\in\R^{p+1}$ we have that $x\E_A\left[AA^T\|B\right]x^T\le x\E_A\left[(1+\tau)A_{w^*}A_{w^*}^T\|B\right]x^T$ a.s.. Therefore
	\begin{align*}
		\E_A\left[vAA^Tv^T\right]
		&=
		\E_A\left[\E_A\left[vAA^Tv^T\|B\right]\right]
		=
		\E_A\left[v\E_A\left[AA^T\|B\right]v^T\right]
		\le
		(1+\tau)\E_A\left[v\E_A\left[A_{w^*}A_{w^*}^T\|B\right]v^T\right]
		\\
		&=(1+\tau)\E_A\left[\E_A\left[vA_{w^*}A_{w^*}^Tv^T\|B\right]\right]
		=\E_A\left[(1+\tau)vA_{w^*}A_{w^*}^Tv^T\right],
	\end{align*}
	i.e. $\sup_{F_A\in C^{1+\tau}_{w^*}} \E\left[vAA^Tv^T\right]\le \E\left[(1+\tau)vA_{w^*}A_{w^*}^Tv^T\right]$. Since $F_{\sqrt{(1+\tau)}A_{w^*}}\in C_{w^*}^{1+\tau}\subseteq C^{1+\tau}$ it follows that 
	$$\sup_{F_A\in C^{1+\tau}} \E\left[vAA^Tv^T\right]
	=
	\sup_{F_A\in C^{1+\tau}_{w^*}} \E\left[vAA^Tv^T\right]
	=\E\left[(1+\tau)vA_{w^*}A_{w^*}^Tv^T\right].$$
	Coming back to \eqref{supeq1} we find
	\begin{align}\label{supeq}
		\sup_{F_A\in C^{1+\tau}_w}R_{A}(\beta)
		&=
		\E\left[v\epsilon\epsilon^Tv^T\right]+(1+\tau)\E\left[vA_{w^*} A_{w^*}^Tv^T\right]
		\nonumber
		\\
		&=R_O(\beta)
		+
		(1+\tau)\E\left[v(A_{w^*}+\epsilon) (A_{w^*}+\epsilon)^Tv^T\right]
		-
		(1+\tau)\E\left[v\epsilon \epsilon^Tv^T\right]
		\nonumber
		\\
		&=R_{O}(\beta)+(1+\tau)R_{A_{w^*}}(\beta)-(1+\tau)R_{O}(\beta)
		\nonumber
		\\
		&=
		(1+\tau)R_{A_{w^*}}(\beta)-\tau R_{O}(\beta).
	\end{align}
	Since
	\begin{align*}
		&R_{A_{w^*}}(\beta)=\E\left[\left(X^{A_{w^*}}\beta-Y^{A_{w^*}}\right)^2\right]=\E\left[v\left(\epsilon+\sum_{i=1}^kw^*_iA_i\right)\left(\epsilon+\sum_{i=1}^kw^*_iA_i\right)^Tv^T\right]
		=\sum_{i=1}^k(w^*_i)^2R_{A_i,l}(\beta),
	\end{align*}
	we may plug this back into \eqref{supeq} and get
	$$ \sup_{F_A\in C^{1+\tau}_{w^*}}R_{A}(\beta)=(1+\tau)\sum_{i=1}^k(w^*_i)^2R_{A_i}(\beta)-\tau R_{O}(\beta)=\frac{1}{2}R_+(\beta)+\frac{1+2\tau}{2}R_\Delta(\beta).$$

	As for the second claim we proceed with a proof by contradiction. Suppose
	$$\sup_{F_A\in C^{1+\tau}}R_{A}(\beta)\not=\sup_{w\in\mathcal{W}}\sup_{F_A\in C_w^{1+\tau}}R_{A}(\beta).$$
	Analogously to \eqref{supeq1} we have that 
	$$\sup_{F_A\in C_w^{1+\tau}}R_{A}(\beta)= u\E\left[\epsilon\epsilon^T\right] u^T
	+
	\sup_{F_A\in C_w^{1+\tau}}u\E_A\left[AA^T\right]u^T=u\E\left[\epsilon\epsilon^T\right] u^T
	+
	u\E\left[A_wA_w^T\right]u^T,$$
	so by assumption $\sup_{w\in\mathcal{W}}u\E\left[A_wA_w^T\right]u^T\not=u\E\left[A_{w^*}A_{w^*}^T\right]u^T$
	and this directly contradicts the definition if $w^*$. If we let 
	$$g(w)=(1+\tau)\sum_{i=1}^k(w_i)^2R_{A_i}(\beta)-\tau R_O(\beta)=\frac{1}{2}R^w_+(\beta)+\frac{1+2\tau}{2}R^w_\Delta(\beta)$$
	then $g$ is obviously continuous and since $\mathcal{W}$ is compact it attains a maximum on this set. If $R_{A_j}(\beta)>R_{A_i}(\beta)$ for all $i\not= j$ then it is readily seen that $g$ is maximized by setting $w_j=1$ and $w_i=0$ for $i\not=j$. In case there are several environments that are tied for the maximal environmental risk then we can distribute the weights among these environments freely as long as $w\in\mathcal{W}$.
\end{proof}
It is important to note that $w^*$ as defined above depends on $\beta$ and $\gamma$. Sometimes we make the dependence on $\beta$ explicit by writing $w^*(\beta)$. Now we may introduce the (set of) worst-risk minimizers.  With the notation from Proposition \ref{optimalsol}, let 
\begin{align}\label{betaopt}
	\mathcal{B}_\gamma=\arg\min_{\beta\in\R^p}R^{w^*(\beta)}_+(\beta)+\gamma R^{w^*(\beta)}_\Delta(\beta)
\end{align}
denote the set of argmin solutions. Note that in general there may be several solutions to \eqref{betaopt}, which is why \eqref{betaopt} is a set.

\section{Estimation of worst-risk minimizer}
\label{sec:estimation}
We now turn to the problem of estimating $\beta_\gamma$. First we must set up a framework for how we handle samples from multiple environments. We will assume that all targets and covariates in the population setting have square integrable components (so that the worst-risk decomposition applies). All of our samples are assumed to live on the same probability space $(\Omega,\mathcal{F},\P)$, although we denote this the same as in the population case for notational convenience, they need not be the same. We assume that for each environment $i$ we have an i.i.d. sequence $\{(Y_u^{A_i},X_u^{A_i})\}_{u=1}^\infty$ where $(Y_u^{A_i},X_u^{A_i})$ is distributed according $(Y^{A_i},X^{A_i})$ and we make no assumption of any SEM relationship for these samples. We will neither make any sort of assumption regarding the dependence structure between the sequences $\{(Y_u^{A_i},X_u^{A_i})\}_{u=1}^\infty$ corresponding to the different environments (regardless of whether there is a specific type of dependence between the environments in the population model). To clarify, we only observe $\{(Y_u^{A_i},X_u^{A_i})\}_{u=1}^\infty$, $i=0,...,k$ and not any shifts or noise. 

Suppose the sample size is $\textbf{n}=\left\{n_{A_0},...,n_{A_k}\right\}$, let $\mathbb{X}^{A_i}(\textbf{n})$ be the $n_{A_i}\times p$ matrix with rows $X^{A_i}_1,...,X^{A_i}_{n_{A_i}}$ (from top to bottom) and similarly let $\mathbb{Y}^{A_i}(\textbf{n})$ be the $n_{A_i}\times 1$ column vector with entries $Y^{A_i}_1,...,Y^{A_i}_{n_{A_i}}$ (from top to bottom). For a shifted environment $A_i$, $1\le i\le k$, we shall also denote the plug-in estimator for the corresponding risk function,
$$\hat{R}_{A_i}(\beta)=\frac{1}{n_{A_i}}\n \mathbb{Y}^{A_i}(\textbf{n})- \beta^T\mathbb{X}^{A_i}(\textbf{n})\n_2^2 $$
Consider the set of corresponding plug-in estimators to \eqref{betaopt},
\begin{align}\label{plugin}
	&\hat{B}(\textbf{n})=\arg\min_{\beta\in\R^p}\hat{R}^{\hat{w}^*(\beta),\beta}_+(\beta)+\gamma \hat{R}^{\hat{w}^*(\beta),\beta}_\Delta(\beta),
\end{align}
where
$$\hat{R}_+^{w,\beta}(\textbf{n})=\sum_{i=1}^kw_i^2\frac{1}{n_{A_i}}\n \mathbb{Y}^{A_i}(\textbf{n})- \beta^T\mathbb{X}^{A_i}(\textbf{n})\n_2^2+\frac{1}{n_O} \n \mathbb{Y}^{O}(\textbf{n})- \beta^T\mathbb{X}^{O}(\textbf{n})\n_2^2,$$
$$\hat{R}_\Delta^{w,\beta}(\textbf{n})=\sum_{i=1}^kw_i^2\frac{1}{n_{A_i}}\n \mathbb{Y}^{A_i}(\textbf{n})- \beta^T\mathbb{X}^{A_i}(\textbf{n})\n_2^2-\frac{1}{n_O} \n \mathbb{Y}^{O}(\textbf{n})- \beta^T\mathbb{X}^{O}(\textbf{n})\n_2^2$$
and $\hat{w}^*(\beta)$ are chosen as the weights $w$ that maximize $\hat{R}^{w,\beta}_+(\beta)+\gamma \hat{R}^{w,\beta}_\Delta(\beta)$. Note that there may be several plug-in estimators, stemming from the fact that there may be several solutions to \eqref{plugin}. The main result of this paper is to provide an explicit and consistent (except for a null set of choices of parameters) estimator to \eqref{betaopt}. Note that \eqref{plugin} does not give an explicit solution, and simply showing consistency for these estimators still leaves the problem of having an unbounded search space ($\R^p$). We will provide a constructive approach that will resolve this issue.

\subsection{Consistency of a constructive estimator}
Denote 
$a_i(l_1,l_2)=\E\left[X^{A_i}(l_1)X^{A_i}(l_2)\right]$, for $l_1\not=l_2$, $a_i(l)=a_i(l,l)=\E\left[\left(X^{A_i}(l)\right)^2\right]$, $b_i(l)=\E\left[X^{A_i}(l)Y^{A_i}\right]$ and $c_i=\E\left[\left(Y^{A_i}\right)^2\right]$. These are the parameters which we will use to identify the points in our parameter space which we define as follows.
Let $\Theta=(\R^+)^p\times(\R^+)^p\times\R^p\times\R^p\times\R^{p\times(p-1)}\times\R^{p\times(p-1)}\times\R^+\times \R^+\subset \R^{2(p^2+p+1)}$ denote the possible parameter space for each environment, so that for environment $i$ we associate
$$\theta_i=\left(\{a_i(l)\}_1^p,\{b_i(l)\}_1^p,\{a_i(l_1,l_2)\}_{l_1\not=l_2, 1\le l_1\le p,1\le l_2\le p},c_i\right)$$
and similarly for a pairing of two environments, $(i,j)$, $i\not=j$ we can associate an element $\theta_{i,j}\in \Theta\times\Theta$ defined by
$$\theta_{i,j}=\left(\{a_i(l)\}_1^p,\{a_j(l)\}_1^p,\{b_i(l)\}_1^p,\{b_j(l)\}_1^p,\{a_i(l_1,l_2)\}_{l_1\not=l_2, 1\le l_1\le p,1\le l_2\le p},\{a_j(l_1,l_2)\}_{l_1\not=l_2, 1\le l_1\le p,1\le l_2\le p},c_i,c_j\right).$$
Let $\dist(x,E)=\inf\{\n y-x\n_2: y\in E\}$ for $x\in\R^m$, $E\subset\R^m$ for some $m\in\N$ and where $\n.\n_2$ denotes the Euclidean norm. Define
$$\widehat{C}^{i,j}_u(\lambda)=\hat{b}_i(u)-\lambda\left(\hat{b}_i(u)-\hat{b}_j(v)\right), $$
$$\widehat{M}^{i,j}_{l,l}(\lambda)=\hat{a}_i(l)-\lambda\left(\hat{a}_i(l)-\hat{a}_j(l)\right),$$
and when $u\not=v$
$$\widehat{M}^{i,j}(\lambda)_{u,v} =\hat{a}_i(u,v)-\lambda\left(\hat{a}_i(u,v)-\hat{a}_j(u,v)\right).$$
Let $\beta(\lambda)=(\hat{M}^{i,j})^{-1}(\lambda)\hat{C}^{i,j}_\lambda$ and $\hat{\mathcal{R}}$ denote the real roots of $\hat{g}(\beta(\lambda))=0$, where $\hat{g}(\beta)=\hat{R}_{A_i}(\beta)-\hat{R}_{A_j}(\beta)$. We also let $\hat{\Lambda}_{\textbf{n}}$ denote the roots of $\det\left(\widehat{M}^{i,j}(\lambda)\right)=0$. The following (a.s. finite for sufficiently large $\textbf{n}$) sets will be used for estimating the solution to \eqref{betaopt}. Denote, 
\begin{eqnarray*}
	\hat{B}_{\mbox{\scriptsize inf}}(\textbf{n})&=&\bigcup_{i=1}^k\left\{\left(\mathbb{G}_+^i+\gamma \mathbb{G}_\Delta^i\right)^{-1}\left(\mathbb{Z}_+^i+\gamma \mathbb{Z}_\Delta^i\right)\right\},\\
	\hat{B}_{\mbox{\scriptsize int}}(\textbf{n})&=&\bigcup_{1\le i<j\le k}\hat{B}^{i,j}(\textbf{n}),
\end{eqnarray*}
where Let $G^i=\E\left[(X^{A_i})^2\right]$, $G_Y^i=\E\left[(Y^{A_i})^2\right]$, $\mathbb{G}_Y^i=\frac{1}{n_{A_i}}\sum_{l=1}^{n_{A_i}}(Y^{A_i}_l)^2$, $Z^i=\E\left[X^{A_i}Y^{A_i}\right]$ and  $\mathbb{Z}^i=\frac{1}{n_{A_i}}\sum_{l=1}^{n_{A_i}}X^{A_i}_lY^{A_i}_l$,
$\mathbb{G}^i_+(\textbf{n})=\frac{1}{n_{A_i}}(\mathbb{X}^{A_i})^T\mathbb{X}^{A_i}+\frac{1}{n_O} (\mathbb{X}^{O})^T\mathbb{X}^{O},$ $\mathbb{G}^i_\Delta(\textbf{n})=\frac{1}{n_{A_i}}(\mathbb{X}^{A_i})^T\mathbb{X}^{A_i}-\frac{1}{n_O} (\mathbb{X}^{O})^T\mathbb{X}^{O},$
$\mathbb{G}^i(\textbf{n})=\frac{1}{n_{A_i}}(\mathbb{X}^{A_i})^T\mathbb{X}^{A_i},$
$\mathbb{G}_Y^i=\frac{1}{n_{A_i}}\sum_{l=1}^{n_{A_i}}(Y^{A_i}_l)^2,$
$\mathbb{Z}^i_+(\textbf{n})=\frac{1}{n_{A_i}}(\mathbb{X}^{A_i})^T\mathbb{Y}^{A_i}+\frac{1}{n_O} (\mathbb{X}^{O})^T\mathbb{Y}^{O},$
$\mathbb{Z}^i_\Delta(\textbf{n})=\frac{1}{n_{A_i}}(\mathbb{X}^{A_i})^T\mathbb{Y}^{A_i}-\frac{1}{n_O} (\mathbb{X}^{O})^T\mathbb{Y}^{O}$, and $\mathbb{Z}^i=\frac{1}{n_{A_i}}\sum_{l=1}^{n_{A_i}}X^{A_i}_lY^{A_i}_l.$ Let 
\begin{align}\label{Bij}
	\hat{B}^{i,j}(\textbf{n})=\left\{\hat{M}^{-1}(\lambda^{i,j})\hat{C}^{i,j}(\lambda^{i,j}): \lambda^{i,j}=\arg\min_{\lambda\in\hat{\mathcal{R}}\setminus\hat{\Lambda}_{\textbf{n}} }\hat{R}_{A_i}\left((\hat{M}^{i,j})^{-1}(\lambda)\hat{C}^{i,j}_\lambda\right)\right\},
\end{align}
if $\hat{R}_{A_i}$ and $\hat{R}_{A_j}$ intersect otherwise we set $\hat{B}^{i,j}(\textbf{n})=\emptyset$. 
Our first main result states that \eqref{betaopt} has a unique solution (outside a set of measure zero in $\Theta\times\Theta$) and that \eqref{plugin} are consistent estimators. We again highlight that fact that it also circumvents the issue of the infinite search space $\mathbb{R}^p$ by using a finite set of candidates in $\hat{B}(\textbf{n})$.  Let 
$$\hat{f}(\beta)=\left((1+\tau)\hat{R}_{A_1}(\beta)-\tau \hat{R}_O(\beta)\right)\vee...\vee \left((1+\tau)\hat{R}_{A_k}(\beta)-\tau \hat{R}_O(\beta)\right)$$
and
$$\hat{f}_i(\beta)=\left((1+\tau)\hat{R}_{A_i}(\beta)-\tau \hat{R}_O(\beta)\right).$$
\begin{thm}\label{optimalregthm}
	For every pair of environments $(i,j)$, $i\not=j$, outside a subset of Lebesgue measure zero $N\in\Theta\times\Theta$ of choices of $\theta_{i,j}$, we have that $|\mathcal{B}_\gamma|=1$, i.e., there is a unique solution to \eqref{betaopt}. Furthermore,  $$\dist\left(\hat{\beta}_\gamma(\textbf{n}),\mathcal{B}_\gamma\right)\xrightarrow{a.s.}0,$$ for any $\hat{\beta}_\gamma(\textbf{n})\in\hat{B}(\textbf{n})$ as $n_{A_0}\wedge...\wedge n_{A_k}\to\infty$, outside $N$. Moreover, for large $\textbf{n}$
\begin{align}\label{betaeqemp}
	&\hat{B}(\textbf{n})=\arg\min_{\beta\in \left(\hat{B}_{inf}(\textbf{n}) \cup \hat{B}_{int}(\textbf{n})\right)\cap\{\beta\in\R:\hspace{1mm} \exists 1\le i \le k, \hat{f}(\beta)= \hat{f}_i(\beta)\}} \hat{f}(\beta),
\end{align}
	
\end{thm}
The proof, together with preliminary lemmas can be found in the Appendix. As this proof is rather technical and lengthy we will now give a brief summary of it.
\begin{proof}[of Theorem \ref{optimalregthm} (broad outline; full proof in appendix)]
	The first part of the proof is to show existence and uniqueness of the minimizer. We start with observing that regardless of how $w^*$ is chosen (defined in Proposition \ref{optimalsol}),
	\begin{eqnarray}
		R_+^{w^*(\beta)}(\beta) +\gamma R_\Delta^{w^*(\beta)}(\beta)&=&\left((1+\tau)R_{A_1}(\beta)-\tau R_O(\beta)\right)\vee...\vee \left((1+\tau)R_{A_k}(\beta)-\tau R_O(\beta)\right) \nonumber \\
		&:=&h_1(\beta)\vee...\vee h_k(\beta). \label{Riskdummy}
	\end{eqnarray}
	Which is the maximum of functions that are ``usually'' (outside a set of measure zero in $\Theta$) strictly convex, which then implies that the maximum is also strictly convex. This will imply uniqueness once we establish existence. To show existence of the minimizer we first show that if $\n\beta\n\to\infty$ in \eqref{Riskdummy}, then the risk diverges to $\infty$ and therefore the infinum must be attained in a compact set and by continuity this will imply that a minimum is attained (i.e. the $\arg\min$ exists). 
	
	We then move on to the structure of the minimizer and its estimator. The first observation we make is that the minimum can only occur at either an inflexion point for some $h_i$ or at some intersection point between some $h_i$ and $h_j$. The set of inflexion points is obviously finite, the potential minima along the intersections is a bit more subtle. We use the method of Lagrange multipliers to tackle this task. Noting that $h_i$ and $h_j$ intersect at some point if and only if $R_{A_i}$ and $R_{A_j}$ intersect at that same point, the Lagrangian is given by (for fixed $i\not=j$),
	$$ \mathcal{L}(\beta,\lambda)=R_{A_i}(\beta)-\lambda g(\beta)$$
	where $g(\beta)=R_{A_i}(\beta)-R_{A_j}(\beta)$. Solving $\nabla_\beta  \mathcal{L}(\beta,\lambda)=0$ leads to (see the formal proof in the Appendix for the definitions of $M^{i,j}$ and $C^{i,j}$)
	\begin{align}\label{Lineq}
		M^{i,j}(\lambda)\beta(\lambda)=C^{i,j}_{\lambda},
	\end{align}
	where we write $\beta(\lambda)$ only to signify that $\beta$ depends on $\lambda$. In order for this approach to be meaningful, the above equation must only have a finite set of solutions (in terms of $\beta$). An infinite set of solutions can only arise when $M^{i,j}(\lambda)$ is rank deficient. We therefore show that for any $\lambda\in\R$ such that $\det\left(M^{i,j}(\lambda)\right)=0$, there are no solutions to \eqref{Lineq} outside a set of measure zero of parameter choices in $\Theta\times\Theta$.  By Lemma \ref{matrixLemma} (in the Appendix) outside a set of measure zero $N_2\in\Theta\times\Theta$, $\mathrm{rank}(M^{i,j}(\lambda))=p-1$ for $\lambda\in\Lambda$ and 
	\begin{align}\label{linearcombTMP}
		&M^{i,j}_{p,.}(\lambda)=\sum_{u=1}^{p-1}s_u(\theta,\lambda)M_{u,.}(\lambda)
	\end{align}
	with all $s_u(\theta,\lambda)\not=0$. It turns out that outside a zero set in $\Theta\times\Theta$, $\{s_u(\theta,\lambda)\}_{u=1}^{p-1}$ are all rational functions on $\R\times\Theta\times\Theta$. Applying the same row operations to the vector $C^{i,j}(\lambda)$ as $M^{i,j}(\lambda)$ we now have that $M^{i,j}(\lambda)\beta=C^{i,j}(\lambda)$ has no solutions if
	$$\left(C^{i,j}_{p}(\lambda)-\sum_{v\not=p}s_v(\theta,\lambda) C^{i,j}_{v}(\lambda)\right)\not=0.$$
	Let $G$ denote the lowest common denominator of the terms in $\left(C^{i,j}_{p}(\lambda)-\sum_{v\not=p}s_v(\theta,\lambda) C^{i,j}_{v}(\lambda)\right)$ then 
	$$Q(\theta,\lambda)= G(\theta,\lambda)\left(C^{i,j}_{p}(\lambda)-\sum_{v\not=p}s_v(\theta,\lambda) C^{i,j}_{v}(\lambda)\right)$$ 
	is a polynomial on $\R\times\Theta\times\Theta$. In order to show that $Q$ has no roots in $\Lambda$, we use that this is equivalent to showing that $P$ and $Q$ have no common roots which is in turn equivalent to showing that their resultant is not zero.

	We then also need to show that are only a finite set of $\lambda$s that are viable candidates, these are the ones that fulfill $g(\beta(\lambda))=0$. As $g(\beta(\lambda))$ has the same roots in terms of $\lambda$ as $\det\left(M^{i,j}(\lambda)\right) g(\beta(\lambda))$, which is a polynomial in terms of $\lambda$ it can only have a finite set of roots unless it is the zero polynomial.  
	
	The next part of the proof considers the plug-in estimators for the candidate points along the intersections and their consistency. We argue analogously to earlier to show that there are only finitely many candidates that solve the empirical corresponding equation to \eqref{Lineq}
	\begin{align}\label{LineqHat}
		\hat{M}^{i,j}(\lambda)\beta(\lambda)=\hat{C}^{i,j}(\lambda),
	\end{align}
	in terms of $\beta(\lambda)$ that also solve $\hat{g}(\beta(\lambda))=0$.
\end{proof}

\subsection{An explicit, consistent approximation}
From a practitioners point of view there is an issue with the estimators in the above theorem. Namely it requires the computation of the roots of a polynomial of degree $p$. By the Abel-Ruffini theorem we cannot solve general polynomials of this type in terms of radicals. There is however an approximate estimator (or in reality a family of estimators) that will also be consistent, which do not require us to compute the roots analytically. The solution lies in the proof of Theorem \ref{optimalregthm}. When computing the roots in \eqref{groots} we note that they are simple roots so we may approximate these roots by using the bisection method.

Let $\{c_m\}_m$ be any sequence in $\N$ such that $c_n\to\infty$, assume we have $\textbf{n}=(n_O,n_{A_1},...,n_{A_k})$ samples from every environment and define $n=n_O\wedge n_{A_1}\wedge...\wedge n_{A_k}$. Let $\tilde{P}$ and $\hat{\tilde{P}}$ be as in the proof of Theorem \ref{optimalregthm} and let $\bar{\beta}_\gamma(\lambda)$ be the approximation of $\hat{\beta}_\gamma(\lambda)$ where we replace the roots of $\hat{\tilde{P}}$ with the following approximation. Since $\tilde{P}(\lambda)$ is a polynomial of degree $p-1$ we may write $\tilde{P}(\lambda)=\sum_{u=0}^{p-1}e_u\lambda^u$ it then follows that $\hat{\tilde{P}}(\lambda)=\sum_{u=0}^{p-1}\hat{e}_u\lambda^u$, where $\hat{e}_u$ is the plug-in estimator of $e_u$. By the Lagrange bound all real roots of $\hat{\tilde{P}}$ can be contained in $\left[-\left(1\vee\sum_{u=1}^{p-2}\left|\frac{\hat{e}_{u}}{\hat{e}_{p-1}}\right|\right),1\vee\sum_{u=1}\left|\frac{\hat{e}_{u}}{\hat{e}_{p-1}}\right|\right]$. Let $R_n=1\vee\sum_{u=1}^{p-2}\left|\frac{\hat{e}_{u}}{\hat{e}_{p-1}}\right|$, then since $\hat{e}_{u}\to e_u$ by the law of large numbers for $u=0,...,p-1$ we have that $R_n\xrightarrow{a.s.}R$ where $R=1\vee\sum_{u=1}^{p-2}\left|\frac{e_{u}}{e_{p-1}}\right|$. So for large $\textbf{n}$, $[-R_n,R_n]\subset[-(R+1),R+1]$. By Theorem 1 in \cite{rump} we have that the minimal distance between any roots of $\tilde{P}$ is bounded below by 
$$\Delta:= \left(1\vee\left|e_{p-1}\right|\right)^{(p-1)(\ln (p-1)+1)}D(\tilde{P})\frac{(2(p-1))^{p-2}}{s^{(p-1)(\ln(p-1)+3)}},$$
where $D(\tilde{P})$ denotes the discriminant of $P$ and $s=\sum_{u=0}^{p-1}|e_u|$. Similarly we let 
$$\hat{\Delta}= \left(1\vee\left|\hat{e}_{p-1}\right|\right)^{(p-1)(\ln (p-1)+1)}D(\hat{\tilde{P}})\frac{(2(p-1))^{p-2}}{\hat{s}^{(p-1)(\ln(p-1)+3)}}$$
then clearly $\hat{\Delta}\xrightarrow{a.s.}\Delta$ by continuity and the law of large numbers.  Let $\epsilon>0$, we now divide $[-R_n,R_n]$ into $m_n:=2\left \lceil{\frac{R_n}{\hat{\Delta}+\epsilon}}\right \rceil $ intervals $\{I_u\}_{u=1}^{m_n}$ of equal length such that $|I_u|<\Delta$. Any such interval $I_u$ has a root to $\hat{\tilde{P}}$ if and only if the two endpoints of $I_u$ are of opposite sign. Consider the $v$:th interval with a sign change for $\hat{\tilde{P}}$, this interval must contain $\hat{\lambda_v}$ (the $v$:th smallest real root to $\hat{\tilde{P}}$), we compute $c_n$ number of bisections to get our approximation $\bar{\hat{\lambda}}_v$ of $\hat{\lambda}$, which then will have the property $\left| \bar{\hat{\lambda}}_u - \hat{\lambda}_u\right|<\frac{1}{2^{c_n}}$ and therefore $\bar{\hat{\lambda}}_v\xrightarrow{a.s.}\lambda_v$. We summarize our findings above in the following theorem.
\begin{thm}
	The bisection estimator $\bar{\beta}_\gamma(\lambda)$ described above has the same consistency property as $\hat{\beta}_\gamma$ in Theorem \ref{optimalregthm} in the sense that
	$$\dist\left(\hat{\beta}_\lambda(\textbf{n}),\mathcal{B}_\gamma\right)\xrightarrow{a.s.}0.$$
	Moreover we will compute at most $(p-1)(\textbf{n}(1)\wedge\textbf{n}(2))$ bisections in total when we have $\textbf{n}$ samples.
\end{thm}

Algorithm~\ref{alg:bisection} summarizes exactly how to compute the approximate estimator given $n_{A_i}$ samples from environment $1\le i\le k$.

\begin{algorithm}\label{alg:bisection}
	Algorithm for computing the approximate estimator 
	\begin{itemize}
		\item
		Compute the inflexion point of to every $R_{A_i}(\beta)$ which is given by $\left(G^i_++\gamma G^i_\Delta\right)^{-1}\left(Z^i_++\gamma Z^i_\Delta\right)$.
		\item
		For every $1\le i<j\le k$ such that $\mathbb{R}_{A_i}(\beta)$ and $\mathbb{R}_{A_j}(\beta)$ intersect, compute all the roots to 
		$$\hat{\tilde{P}}(\lambda)=\det\left(\hat{M}^{i,j}(\lambda)\right)^2\hat{g}(\beta(\lambda)),$$ 
		(where $\hat{g}(\beta(\lambda))$ is given by \eqref{groots}) using the bisection method outlined above. For every such root $\lambda$, compute $\beta(\lambda)=\left(M^{i,j}\right)^{-1}(\lambda)C^{i,j}(\lambda)$
		\item
		The $\arg\min$ solution is now given by the $\beta$ amongst the ones computed above that minimizes 
		$$(1+\tau)\left(\hat{R}_{A_1}(\beta)\vee...\vee \hat{R}_{A_k}(\beta)\right)-\tau \hat{R}_O(\beta).$$
	\end{itemize}
\end{algorithm}

\section{Conclusions}
In this manuscript, we derived the optimal out-of-sample risk-minimizing predictor for a specific target within a nonlinear system. The system is  trained across various within-sample environments, consisting of an observational and several shifted environments, each corresponding to a nonlinear structural equation model (SEM) with its unique shift and noise vector, both in $L^2$. Unlike previous methodologies, we also allowed for shifts in the target value. We established a sieve of out-of-sample environments, encompassing all shifts $\tilde{A}$ that are at most $\gamma$ times as strong as any weighted average of observed shift vectors. For each $\beta\in\R^p$, we demonstrate that the supremum of the risk functions $R_{\tilde{A}}(\beta)$ can be decomposed into a (positive) non-linear combination of risk functions. Subsequently, we defined the risk minimizer, $\beta_\gamma$, as the argument $\beta$ that minimizes this risk. The main finding of this paper is that the risk minimizer can be consistently estimated using an estimator outside a set of zero Lebesgue measure in the parameter space. An inherent challenge in such estimation lies in solving a general-degree polynomial, lacking an explicit solution. We resolved this by establishing an approximate estimator, that is consistent, employing the bisection method.

\bibliographystyle{biometrika}
\bibliography{bibliopaper}

\newpage
\appendix

	
	\section{Proof of Theorem 1}
	Before the proof of Theorem \ref{optimalregthm} we will need to construct a sizeable toolbox in the form of several Lemmas that will now follow.
	For each $\theta\in\Theta$ we may bijectively associate a pairing of an affine (and symmetric) matrix function and an affine row vector as follows.
	\begin{def}\label{parammatrix}
		For every $\theta_{i,j}\in\Theta\times\Theta$ let the \textit{affine covariate matrix},  $M^{i,j}(\lambda)$ and the \textit{affine target vector} $C^{i,j}(\lambda)$ be defined element wise by
		$$C^{i,j}(\lambda)_u=b_i(u)-\lambda\left(b_i(u)-b_j(u)\right),$$
		$$M^{i,j}(\lambda)_{l,l} =a_i(u)-\lambda(a_i(u)-a_j(u))).$$
		When $u\not=v$
		$$M^{i,j}(\lambda)_{u,v} =a_i(u,v)-\lambda \left( a_i(u,v)-a_j(u,v)\right).$$
	\end{def}
	We may regard $\det\left(M^{i,j}(\lambda)\right)$ as polynomial in $\R$ with coefficients in $\Theta\times\Theta$. Doing so we will denote the roots in terms of $\lambda$ as $\Lambda(\theta_{i,j})$ where we highlight the dependence on $\theta_{i,j}\in\Theta\times\Theta$. We will denote the real roots of $\det\left(M^{i,j}(\lambda)\right)$ as $\Lambda(\theta_{i,j})_\R$. Often times we will suppress the dependence on $\theta_{i,j}$ for brevity when we see fit.
	The following result comes from complex analysis.
	\begin{lemma}\label{realsimple}
		Let $P(\lambda)=\sum_{u=0}^ma_u\lambda^u$ be a polynomial with real coefficients ($a_m\not=0$) with a real simple root $r$. For any $0<\epsilon<\min_{1\le u<z\le v}|r_u-r_z|$ there exists $\delta>0$ such that if we consider any polynomial of the form $\tilde{P}(\lambda)=\sum_{u=0}^mb_u\lambda^u$ with $|a_u-b_u|<\delta$ then $\tilde{P}$ must have a simple real root in $(r_u-\epsilon,r_u+\epsilon)$.
	\end{lemma}
	The above result is then applied to get the following lemma which is what we actually will need.
	\begin{lemma}\label{simpleroots}
		Let $P(\lambda)=\sum_{u=0}^ma_u\lambda^u$ be a polynomial with real coefficients ($a_m\not=0$) and whose real roots $r_1<...<r_v$ ($v\le m$) are simple. If $P_n(\lambda)=\sum_{u=0}^ma_{u,n}\lambda^u$ is a sequence of polynomials such that $a_{u,n}\to a_u$, $u=0,...,m$ then if we denote the real roots of $P_n(\lambda)$ as $r_1(n)<...<r_{w_n}(n)$ then $w_n=v$ for large enough $n$ and $r_u(n)\to r_u$ for $u=1,...,v$.
	\end{lemma}
	\begin{proof}
		We know by Lemma \ref{realsimple} that for any $0<\epsilon<\min_{1\le u<z\le v}|r_u-r_z|$ there exists $\delta>0$ such that if $|a_u-a_{u,n}|<\delta$ then $P_n$ must have a simple real root in $(r_u-\epsilon,r_u+\epsilon)$, which shows that $w(n)\ge v$ for large enough $n$ and that $r_u(n)\to r_u$ for $u=1,...,v$. If $v=m$ we are done, since then all roots are real and simple. If $v=m$ we are done. Suppose instead that $v<m$. We know that all roots of $P_n$ must converge to those of $P$ so take any root $r$ of $P$ with non-zero imaginary part $c=\mathrm{Im}(r)$ then we know that for large enough $n$, $P_n$ must have a root $r'(n)$ such that $|r-r'(n)|<c/2$ which implies that the imaginary part of $r'(n)$ is non-zero for such $n$. But this is true for all roots of $P$ that are not real and since $P_n$ has the same degree as $P$ for large enough $n$ (when $a_{m,n}\not=0$), $P_n$ must have the same number of roots that are not real as $P$, i.e. $w_(n)=v$ for such $n$.
	\end{proof}
	A result from measure theory which will lie at the heart of the method by which we prove Theorem \ref{optimalregthm} is the following.
	\begin{lemma}\label{PolyLemma}
		A polynomial on $\R^n$, for any $n\in\N$ is either identically zero or is only zero on set of Lebesgue measure zero.
	\end{lemma}
	The following lemma is an immediate application of the one above and illustrates how we will apply the above lemma in this paper.
	\begin{lemma}\label{discriminantLemma}
		Any non-zero polynomial $P_{\theta_{i,j}}\left(\lambda\right)$ on $\R\times\Theta\times\Theta$ has simple roots (in terms of $\lambda$) outside a set of measure zero in $\Theta\times\Theta$.
	\end{lemma}
	\begin{proof}
		$P$ only has simple roots if the discriminant is non-zero. The discriminant is a (non-zero) polynomial on $\Theta\times\Theta$ and therefore is only zero on a set of measure zero in $\Theta$.
	\end{proof}
	We will from now on omit the dependence on $\theta_{i,j}$ of polynomials of the kind above. The reason being that we regard it as polynomial on $\R$ with coefficients in $\Theta\times\Theta$ and we are interested in the roots in terms of $\lambda$.
	\begin{lemma}\label{matrixLemma}
		The following results holds for the affine parameter matrix $M^{i,j}(\lambda)$ as defined in \ref{parammatrix}.
		\begin{itemize}
			\item [1] Outside a set $N$ of measure zero in $\Theta$, there are at most $p$ elements in $\Lambda$ and for any $\lambda\in\Lambda$,
			$rank\left(M^{i,j}(\lambda)\right)=p-1$.
			\item[2] Moreover for any such $\lambda$ there is a unique set of $p-1$ real numbers $s_1(\theta,\lambda),...,s_{p-1}(\theta,\lambda)$, all of which must be non-zero and such that 
			\begin{align}\label{coeffrepreq}
				&M^{i,j}_{p,.}(\lambda)=\sum_{u=1}^{p-1}s_u(\theta,\lambda)M^{i,j}_{u,.}(\lambda),
			\end{align}
			outside $N$.
		\end{itemize}
		
	\end{lemma}
	\begin{proof}
		Denote $P(\lambda)=\det\left(M^{i,j}(\lambda)\right)$, by the fundamental theorem of algebra $P$ has at most $p$ distinct solutions in $\mathbb{C}$. The rank of $M^{i,j}(\lambda)$ is $p-1$ if and only if there exists some non-zero minor, that is to say there exists a minor that has no roots in common with $P$. Consider a submatrix $M_{u,v}(\lambda)$ (where the subscripts $u,v$ indicate that we remove row $u$ and column $v$) of $M^{i,j}(\lambda)$, and let $P_{u,v}(\lambda)=\det\left(M_{u,v}(\lambda)\right)$. We wish to show that outside a set of measure zero in $\Theta\times\Theta$, there exists a minor such that $P$ and $P_{u,v}$ has no roots in common. $P$ and $P_{u,v}$ has a root in common root in $\mathbb{C}$ if and only if their resultant is zero (since $\mathbb{C}$ is an algebraically closed field). But the resultant is a polynomial on $\Theta\times\Theta$ so by Lemma \ref{PolyLemma} this polynomial is either the zero polynomial or it is zero on a set of measure zero in $\Theta\times\Theta$. So to prove [1], it suffices to show that there exists any $\theta\in\Theta\times\Theta$ such that the resultant of $P$ and $P_{u,v}$ is not zero. Recall however that this is true if and only if for some $\theta$, $P(\lambda)=0$ implies $P_{u,v}(\lambda)\not=0$ for some $u,v\in\{1,...,p\}$ and all $\lambda\in\R$. For this purpose consider the matrix $M'(\lambda)$ where 
		%
		\begin{equation*}
			M'(\lambda) = 
			\begin{pmatrix}
				\lambda & \lambda-1 &0 & \cdots & 0 \\
				0 & 1 & 0 & \cdots & 0  \\
				\vdots  & \vdots  & \ddots & \vdots & \vdots \\
				0 & \cdots & 0 & 1 & 0  \\
				0 & \cdots & 0 & 0 & 1
			\end{pmatrix}
		\end{equation*}
		In this case the only root to $P$ is given by $\lambda=0$ while the only root for $P_{1,2}$ is given by $\lambda=1$. This proves the [1]. For the final result, it is enough to show that every first minor of the form $\det\left(M^{i,j}_{u,1}(\lambda)\right)$ for $u=1,...,p-1$ is non-zero since then $M^{i,j}_{p,.}(\lambda)$ must be linearly independent from all subsets of $n-2$ other rows. $\det\left(M^{i,j}_{1,u}(\lambda)\right)=0$ for a root of $P$ if and only if $res(P_{u,v},P)$ is zero. Again it suffices to show that for each $1\le u\le p-1$ there exists any $\theta\in\Theta\times\Theta$ such that $P$ and $P_{u,1}$ have no common root. Letting $\theta$ be such that $M^{i,j}(\lambda)$ is the diagonal matrix with $\lambda$ as diagonal element number $u$ and all other diagonal entries one implies that $P$ only has the root $0$ and $P_{u,1}$ has no roots at all. This completes the proof of the final claim.
	\end{proof}
	\begin{lemma}\label{matrixNoTiesLemma}
		The coefficients $\{s_u(\theta,\lambda)\}_{u=1}^{p-1}$ in Lemma \ref{matrixLemma} are rational functions on $\R\times\Theta\times\Theta$, whose roots in terms of $\lambda$ do not lie in $\Lambda$ outside some zero set in $\Theta\times\Theta$.
	\end{lemma}
	\begin{proof}
		For any root of $\lambda\in\Lambda$, let $T_1(M^{i,j}(\lambda))$ denote the matrix resulting from subtracting $\frac{M^{i,j}(\lambda)(u,1)}{M^{i,j}(\lambda)(1,1)}M^{i,j}_{1,.}$ from row number $u$, $u=2,...,p$ in $M^{i,j}(\lambda)$, for those $\theta\in\Theta\times\Theta$ such that $M^{i,j}(\lambda)(1,1)\not=0$, $\forall \lambda\in\Lambda(\theta)$. For $1<v<p$, let $T_v(M^{i,j}(\lambda))$ denote the matrix resulting from subtracting $\frac{T_{v-1}(M^{i,j}(\lambda))(v,u)}{T_{v-1}(M^{i,j}(\lambda))(v,v)}T_{v-1}(M^{i,j}(\lambda)_{v,.})$ (whenever $T_{v-1}(M^{i,j}(\lambda))(v,v)\not=0$) from row number $u$, $u=v+1,...,p$ in $T_{v-1}(M^{i,j}(\lambda))$. Then outside a zero set in $\Theta\times\Theta$, $T_vM^{i,j}(\lambda)$ is well-defined for $v=1,...,p-1$ and $(T_vM^{i,j}(\lambda))_{p,.}=(0,...,0)$.
		The following procedure shows us that outside a zero set in $\Theta\times\Theta$ we may bring $M^{i,j}(\lambda)$ to a row echelon form in $p-1$ steps, without permuting any rows or columns while making the final row the only zero-row in this matrix. First we see that $d_1(\lambda)=(M^{i,j}(\lambda))(1,1)$ and $P(\lambda)$ only have a common root if their resultant is zero. As before it suffices to find $\theta\in\Theta\times\Theta$ such that $d_1(\lambda)\not=0$ for any $\lambda\in\Lambda(\theta)$. Consider the matrix $\tilde{M}$ with $\tilde{M}(l,l)=1$ for $l=1,...,p-1$, $\tilde{M}(p,p)=p-\lambda$, $\tilde{M}(p,l)=1$ for $l=1,...,p-1$, $\tilde{M}(1,l)=1$ for $l=1,...,p-1$ and zero on all other entries. 
		\begin{equation*}
			\tilde{M}(\lambda) = 
			\begin{pmatrix}
				1 & 0 & \cdots & 0 & 1 \\
				0 & 1 & \cdots & 0 & 1 \\
				\vdots  & \vdots  & \ddots & \vdots & \vdots \\
				0  & 0  & \cdots & 1 & 1  \\
				1 & 1 & \cdots & 1 & \lambda+p 
			\end{pmatrix}
		\end{equation*}
		Then $d_1\equiv1$ and $\Lambda(\theta)=\{0\}$ and this shows that $d_1$ and $P$ do not have common roots outside some zero set $N_1$, so that we may divide by $d_1$. On $N_1$, $\frac{1}{d_1}=\frac{1}{M^{i,j}(1,1)}$ is well defined and in order to clear the first column below row 1 we want to ensure that we do not clear the next element below on the diagonal ($(M^{i,j}(\lambda))(2,2)$) when doing so. We therefore have to check that $\frac{(M^{i,j}(\lambda))(2,1)}{(M^{i,j}(\lambda))(1,1)}(M^{i,j}(\lambda))(1,2)\not=(M^{i,j}(\lambda))(2,2)$, which is implied if $(M^{i,j}(\lambda))(2,1)(M^{i,j}(\lambda))(1,2)-(M^{i,j}(\lambda))(2,2)(M^{i,j}(\lambda))(1,1)\not=0$. $(M^{i,j}(\lambda))(2,1)(M^{i,j}(\lambda))(1,2)-(M^{i,j}(\lambda))(2,2)(M^{i,j}(\lambda))(1,1)$ is a polynomial on $\Theta\times\Theta$ and we wish to show it has no roots in common with $P$ outside a zero set $N_2\in\Theta\times\Theta$. It suffices to find a $\theta\in\Theta\times\Theta\setminus N_1$ such that the two polynomials do not have a root in common. As $\tilde{M}(2,1)\tilde{M}(1,2)-\tilde{M}(2,2)\tilde{M}(1,1)=-1$, the matrix $\tilde{M}$ still suffices for this purpose. This also shows that all entries in $T_1(M^{i,j}(\lambda))$ are rational functions on $\R\times\Theta\times\Theta$. Suppose now that $1<v\le p-1$, $(T_{v-1}(M^{i,j}(\lambda))(v,v)\not=0$, all entries in $(T_{v-1}(M^{i,j}(\lambda))$ are rational on $\R\times\Theta\times\Theta$  and
		\begin{align}\label{dia}
			&\frac{T_{v-1}(M^{i,j}(\lambda))(v,v-1)}{(T_{v-1}(M^{i,j}(\lambda))(v-1,v-1)}T_{v-1}\left(M^{i,j}(\lambda)\right)(v-1,v)\not=T_{v-1}\left(M^{i,j}(\lambda)\right)(v,v),
		\end{align}
		outside a zero set $N_{v-1}\in\Theta\times\Theta$. This means that we have successfully cleared everything below the main diagonal until the $v-1$:th element on the diagonal counting from the upper left corner. The elements of $(T_{v}(M^{i,j}(\lambda)))$ are obviously rational functions on $\R\times\Theta\times\Theta$ due to the induction hypothesis and the fact that the elements in $(T_{v}(M^{i,j}(\lambda)))$ are formed by products of rational functions in $(T_{v-1}(M^{i,j}(\lambda)))$. Moreover there are no poles in $\Lambda$, outside $N_{v-1}$ (this follows from \eqref{dia}), so the roots of $(T_{v}(M^{i,j}(\lambda))(v,v)$ coincides with those of its numerator polynomial, $Q_v$. To show $Q_v$ is not the zero polynomial take the above $\tilde{M}$ again and note that $(T_{v}(\tilde{M})(v,v)=\tilde{M}(v,v)$ for $v=1,...,p-1$. To show that
		$$\frac{T_{v}(M^{i,j}(\lambda))(v+1,v)}{(T_{v}(M^{i,j}(\lambda))(v,v)}T_v(M^{i,j}(\lambda))(v,v+1)\not=T_v(M^{i,j}(\lambda))(v+1,v+1),$$
		for $\lambda\in\Lambda$, it suffices to show that the numerator polynomial of 
		$$T_{v}(M^{i,j}(\lambda))(v+1,v)T_{v}(M^{i,j}(\lambda))(v,v+1)-T_{v}(M^{i,j}(\lambda))(v+1,v+1)T_{v}(M^{i,j}(\lambda))(v,v),$$
		does not have a root in common with $P$, i.e., their resultant is only zero on a zero set. For this purpose we can again recycle $\tilde{M}$
		where we have that 
		$$T_{v}(\tilde{M})(v+1,v)T_{v}(\tilde{M})(v,v+1)-T_{v}(\tilde{M})(v+1,v+1)T_{v}(\tilde{M})(v,v)=-1,$$
		similarly to before, when $v<p$. We then let $N'_{v}$ denote the zero set where either $Q_v(\lambda)=0$ or $\frac{T_{v}(M^{i,j}(\lambda))(v+1,v)}{(T_{v}(M^{i,j}(\lambda))(v,v)}M^{i,j}_{v,v+1}=M^{i,j}(v+1,v+1)$ and then let $N_v=N_1\cup....\cup N_{v-1}\cup N_v'$ which is the desired zero set. Note that there will be two rational functions $e_{p-1},e_p$ on $\R\times\Theta\times\Theta$ such that
		\begin{equation*}
			\left(T_{p-2}(M^{i,j}(\lambda))\right)_{p,.} = 
			\begin{pmatrix}
				0 & \cdots & 0 & e_{p-1} & e_p \\
			\end{pmatrix},
		\end{equation*}
		i.e. the only two possible non-zero elements of $\left(T_{p-2}(M^{i,j}(\lambda))\right)_{p,.}$ are the last two elements. Now we must have that
		\begin{equation*}
			\left(T_{p-1}(M^{i,j}(\lambda))\right)_{p,.} = 
			\begin{pmatrix}
				0 & \cdots & 0 \\
			\end{pmatrix},
		\end{equation*}
		indeed, since if this was not the case we would have that $\mathrm{rank}\left(T_{p-1}(M^{i,j}(\lambda))\right)=p$ (since we have all elements on the main diagonal are non-zero and all elements below it have been cleared), but $\mathrm{rank}\left(T_{p-1}(M^{i,j}(\lambda))\right)=\mathrm{rank}\left(M^{i,j}(\lambda)\right)=p$ since $\lambda\in \Lambda$.
		The fact that $T_{p-1}(M^{i,j}(\lambda))_{p,.}=0$ leads to the following equation.
		\begin{align}\label{Mp}
			&M^{i,j}(\lambda)_{p,.}=\frac{(M^{i,j}(\lambda))(p,1)}{(M^{i,j}(\lambda))(1,1)} M^{i,j}(\lambda)_{1,.}+\frac{T_1(M^{i,j}(\lambda))(p,2)}{T_1(M^{i,j}(\lambda))(2,2)} T_1(M^{i,j}(\lambda))_{2,.}+....
			\nonumber
			\\
			&+\frac{T_{p-2}(M^{i,j}(\lambda))(p,p-1)}{T_{p-2}(M^{i,j}(\lambda))(p-1,p-1)} T_{p-2}(M^{i,j}(\lambda))_{p-1,.}
		\end{align}
		$T_1(M^{i,j}(\lambda))_{2,.}$ is a linear combination of $(M^{i,j}(\lambda))_{1,.}$ and $(M^{i,j}(\lambda))_{2,.}$ with coefficients that are rational on $\Theta\times\R$. Similarly $T_v(M^{i,j}(\lambda))_{v,.}$ is a linear combination of $\{(M^{i,j}(\lambda))_{u,.}\}_{u=1}^v$, for $v=2,...,p-1$, with coefficients that are rational on $\R\times\Theta\times\Theta$. This together with \eqref{Mp} and the uniqueness of the representation \eqref{coeffrepreq} shows that the coefficients $\{s_u\}_{u=1}^{p-1}$ are rational on $\R\times\Theta\times\Theta$.
	\end{proof}
	
	\begin{lemma}\label{coefficientslemma}
		Suppose $A$ is a $p\times p$ matrix of rank $p-1$ such that $A_{p,.}$ is linearly independent of every subset of $p-2$ other rows of $A$. Let $A_{p,.}=\sum_{u=1}^{p-1}c_uA_{u,.}$ be its unique decomposition in $span\{A_{1,.},...,A_{p-1,.}\}$. If $\{A(n)\}_n$ is a sequence of $p\times p$ matrices of rank $p-1$ that also share the property that $A_{p,.}(n)$ is linearly independent of every subset of $p-2$ other rows of $A(n)$ and such that $\lim_{n\to\infty}\n A-A(n)\n=0$ then if we let $A_{p,.}(n)=\sum_{u=1}^{p-1}c_u(n)A_{u,.}(n)$ be its unique decomposition in $span\{A_{1,.}(n),...,A_{p-1,.}(n)\}$   it follows that $c_u(n)\to c_u$
	\end{lemma}
	\begin{proof}
		Let $x=A_{p,.}$ and $V=\left(A_{1,.}^T A_{2,.}^T ... A_{p-1,.}^T\right)$ (which is then a $p\times (p-1)$ matrix) so that $x=VC$ if we let $C=\left(c_1 ... c_{p-1} \right)^T$. Note that $V$ has full column rank which implies $C=V^+x$ is a unique solution for $C$, where $V^+$ denotes the Penrose-inverse. This implies that $ \n C \n_\infty\le \n V^+\n_\infty \n A_{p,.}\n_\infty$. Similarly we get that $ \n C(n) \n_\infty\le \n V(n)^+\n_\infty \n A_{p,.}(n)\n_\infty$ where $V(n)^+$ is the Penrose inverse of the matrix $V(n)=\left(A_{1,.}(n)^T A_{2,.}(n)^T ... A_{p-1,.}(n)^T\right)$ and $C(n)=\left(c_1(n) ... c_{p-1}(n)\right)$. Since all the matrices $\{V(n)\}_n$ have rank $p-1$ and $\n V(n)-V \n_\infty$ by the law of large numbers it follows that $\n V(n)^+- V^+\n_\infty\xrightarrow{a.s.}0$ which implies $\n V(n)^+\n_\infty \xrightarrow{a.s.}\n V^+\n_\infty$. By the law of large numbers we also have that $\n A_{p,.}(n) \n_\infty\xrightarrow{a.s.}\n A_{p,.}\n_\infty$. All of this implies that the sequences $\{c_u(n)\}_n$ are bounded for $u=1,...,p-1$. Suppose now that for some $1\le m< p$, $c_m(n)\not\to c_m$. Then since the sequence $\{c_m(n)\}_n$ is bounded there exists a subsequence $\{c_m(n_k)\}_k$ such that $c_m(n_k)\to c'_m\not=c_u$ and $c_v(n_k)\to c'_v$, $1\le v\le p-1$. Since $A_{v,.}(n)\to A_{v,.}$, $1\le v\le p$ it follows that $A_{p,.}=\sum_{u=1}^{p-1}c'_uA_{u,.}$ but $c'_u\not=c_u$ this contradicts the uniqueness of the decomposition and this gives us the result.
	\end{proof}
	We are now ready to prove the main result of the paper.
	\begin{proof}[of Theorem \ref{optimalregthm}]~~\\
		\subsection*{Part 1: Existence and uniqueness of the minimizer.}
		Regardless of how the weights $w^*(\beta)$ (as long as it is in accordance with the definition of $w^*$ in Proposition \ref{optimalsol}) are distributed among the the risk functions, we have that
		$$R_+^{w^*(\beta)}(\beta) +\gamma R_\Delta^{w^*(\beta)}(\beta)=(1+\tau)\left(R_{A_1}(\beta)\vee...\vee R_{A_k}(\beta)\right)-\tau R_O(\beta).$$
		It will be helpful to re-write the minimizer in the following way
		\begin{align}\label{solution}
			\beta_\gamma&=\arg\min_{\beta\in\R^p}(1+\tau)\left(R_{A_1}(\beta)\vee...\vee R_{A_k}(\beta)\right)-\tau R_O(\beta)\nonumber
			\\
			&=\arg\min_{\beta\in\R^p}\left((1+\tau)R_{A_1}(\beta)-\tau R_O(\beta)\right)\vee...\vee \left((1+\tau)R_{A_k}(\beta)-\tau R_O(\beta)\right).
		\end{align}
		Let $h_i(\beta)=(1+\tau)R_{A_i}(\beta)-\tau R_O(\beta)$ so that $\beta_\gamma=\arg\min_{\beta\in\R}h_1(\beta)\vee...\vee h_k(\beta)$. We also let $f(\beta)=h_1(\beta)\vee...\vee h_k(\beta)$. Note that
		\begin{align*}
			G^i_\Delta
			&=
			\E\left[\left((I-B)^{-1}_{1:p,\textbf{.}}\beta \left(A_i+\epsilon_{X^{A_i}} \right)\right)^T(I-B)^{-1}_{1:p,\textbf{.}}\beta \left(A_i+\epsilon_{X^{A_i}} \right)\right]
			-
			\E\left[\left((I-B)^{-1}_{1:p,\textbf{.}}\beta \epsilon_{X^{O}} \right)^T(I-B)^{-1}_{1:p,\textbf{.}}\beta \epsilon_{X^{O}} \right]
			\\
			&=
			\E\left[\left((I-B)^{-1}_{1:p,\textbf{.}}\beta A_i \right)^T(I-B)^{-1}_{1:p,\textbf{.}}\beta A_i \right]
		\end{align*}
		which implies that $G^i_\Delta$ is positive semi-definite. When $G^i_\Delta$ is full-rank it is also positive definite. This is true if $\det\left(G^i_\Delta\right)\not=0$, but $\det\left(G^i_\Delta\right)$ is a polynomial on $\Theta$ and hence, by Lemma \ref{PolyLemma}, it is zero only on a set of Lebesgue measure zero in $\Theta$. Thus $G^i_\Delta$ is positive definite outside a set of Lebesgue measure zero in $\Theta$, which implies that $h_i$ is strictly convex. The maximum of of strictly convex functions is again strictly convex and due to the fact that a strictly convex function has at most one minimizer, we have uniqueness of the $\arg\min$ whenever it exists. 
		
		We now show that such a minimum exists. The minimum of $f$ can only be achieved at either an inflexion point for some $h_i$ or along the intersection (where $f$ might not be differentiable) of two $h_i$'s. The inflexion point of $h_i(\beta)$ is achieved at $\beta_{inf(i)}:=\left(G^i_++\gamma G^i_\Delta\right)^{-1}\left(Z^i_++\gamma Z^i_\Delta\right)$ (this is readily verified by solving $\nabla h_i(\beta)=0$), whenever $G^i_++\gamma G^i_\Delta$ is full rank. Since $G_\Delta \preceq G_+$, if $G_\Delta$ then so is $G_+$. Hence, if $G_\Delta$ is full rank then so is $G^i_++\gamma G^i_\Delta$ (and hence invertible). We now study any potential minima along the intersection of two risk functions $R_i(\beta)$ and $R_j(\beta)$. 
		Expanding the risk function we see that
		$$R_{A_i}(\beta)=\E\left[(Y^{A_i})^2\right]-2\beta\E\left[Y^{A_i}X^{A_i}\right]+\beta G^i\beta^T. $$
		Since $G^i$ is positive definite outside a set of measure zero, by a diagonalization of $G^i$, $G^i=ODO^T$, where $O$ is an orthogonal matrix and $D$ a diagonal matrix with positive eigenvalues $\lambda_1,...,\lambda_p$ such that $\n \beta O \n=\n\beta\n$. Letting $\tilde{\beta}=\beta O$ (recall that we defined $\beta$ as a row vector and note that $\tilde{\beta}O^T=\beta $), we have that 
		$$\beta G^i\beta^T=\beta ODO^T\beta^T=\tilde{\beta}D\tilde{\beta}^T=\sum_{u=1}^p\lambda_u\tilde{\beta}_u^2.$$
		So if we let $\tilde{\lambda}=\min_{1\le u\le p}\lambda_u$ then
		\begin{align*}
			R_{A_i}(\beta)
			&=
			\E\left[(Y^{A_i})^2\right]-2\beta\E\left[Y^{A_i}X^{A_i}\right]+\beta G^i\beta^T
			=
			\E\left[(Y^{A_i})^2\right]-2\langle\tilde{\beta}O^T,\E\left[Y^{A_i}X^{A_i}\right]\rangle+\sum_{u=1}^p\lambda_u\tilde{\beta}_u^2
			\\
			&\ge \E\left[(Y^{A_i})^2\right]-2\n \tilde{\beta}O^T\n_2\n \E\left[ Y^{A_i}X^{A_i}\right]\n_2+\tilde{\lambda}\n \tilde{\beta}\n_2^2
			\ge
			\E\left[(Y^{A_i})^2\right]-2\n \tilde{\beta}O^T\n_2 \E\left[ \n Y^{A_i}X^{A_i}\n_2\right]+\tilde{\lambda}\n \tilde{\beta}\n_2^2
			\\
			&=\E\left[(Y^{A_i})^2\right] + \n\beta\n_2\left(\tilde{\lambda}\n\beta\n_2-2\E\left[ \n Y^{A_i}\n_2\n X^{A_i}\n_2\right]\right),
		\end{align*}
		where we used Cauchy-Schwarz and Jensen's inequality. This implies that $\liminf_{\n \beta \n\to\infty}R_{A_i}(\beta)=\infty$. Since 
		\\$\inf_{\beta\in\R^p:R_{A_i}(\beta)=R_{A_j}(\beta)} R_{A_i}(\beta)$
		exists and is finite it cannot be a limit where $\n\beta\n\to\infty$ (the risk functions must obviously be finite at any intersection point). Therefore there exists a sequence $\{\beta_l\}_l$ with $\n \beta_l \n\le C$ for some $C<\infty$ such that $\liminf_{\beta\in\R^p:R_{A_i}(\beta)=R_{A_j}(\beta)} R_{A_i}(\beta)=\lim_{l\to\infty}R_{A_i}(\beta_l)$, so in other words
		$$\liminf_{\beta\in\R^p:R_{A_i}(\beta)=R_{A_j}(\beta)} R_{A_i}(\beta)=\liminf_{\beta\in\R^p:R_{A_i}(\beta)=R_{A_j}(\beta), \n \beta\n\le C} R_{A_i}(\beta).$$
		The set $\left\{\beta\in\R^p:R_{A_i}(\beta)=R_{A_j}(\beta)\right\}$ is closed since it is the kernel of the continuous function $R_{A_i}(\beta)-R_{A_j}(\beta)$. Therefore the set $\left\{\beta\in\R^p:R_{A_i}(\beta)=R_{A_j}(\beta), \n \beta\n\le C)\right\}$ is compact but the infinum of a continuous function over a compact set is in fact a minimum. To conclude, if $h_i$ and $h_j$ intersect for some $i\not=j$ then there is a minimum attained along this intersection, which then implies that that \eqref{solution} has a unique solution.
		
		~~\\
		\subsection*{Part 2: Structure of the solution and the estimator.}
		~~\\
		\textit{Step 1: Find and organize all possible candidates for the minimizer:}
		\\
		The plug-in estimator for the inflexion points of $h_i$ is given by $\left(\mathbb{G}^i_++\gamma \mathbb{G}^i_\Delta\right)^{-1}\left(\mathbb{Z}^i_++\gamma \mathbb{Z}^i_\Delta\right)$, by part 1 of this proof the inverse in the population version of this expression is well-defined outside a set of measure zero in $\Theta$. Therefore it follows by continuity and the law of large numbers that this estimator is consistent in the almost sure sense. Let $B_{inf}=\{\beta_{inf(1)},...,\beta_{inf(k)}\}$ denote the set of inflexion points for the different $h_i$s and let $\hat{B}_{inf}(\textbf{n})$ denote the corresponding (plug-in) estimators. Note that $h_i$ and $h_j$ intersect if and only if $R_i$ and $R_j$ intersect. If $R_i$ and $R_j$ intersect let $B^{i,j}$ be its set of argmin points. By part 1 of this proof, this set is either empty or just a singleton. Correspondingly, if $\hat{R}_i$ and $\hat{R}_j$ intersect we let $\hat{B}^{i,j}(\textbf{n})$ denote the set of argmin points along this intersection (this set may contain more than one point!). If $\hat{R}_i$ and $\hat{R}_j$ do not intersect we set $\hat{B}^{i,j}(\textbf{n})=\emptyset$. Let $B_{int}=B^{1,1}\cup...\cup B^{1,k}\cup B^{2,3}\cup...\cup B^{2,k}\cup...\cup B^{k-1,k}\cup B^{k,k}$ denote all the intersections points between the different $h_i$s and let $\hat{B}_{int}(\textbf{n})$ denote the corresponding set of estimators. We also let $B=B_{int}\cup B_{inf}$ and $\hat{B}(\textbf{n})=\hat{B}_{int}(\textbf{n})\cup \hat{B}_{inf}(\textbf{n})$. For $\beta\in B^{i,j}$ to be a candidate for the $\arg\min$ we must have that $h_i(\beta)=f(\beta)$. Let $\hat{h}_i(\beta)=(1+\tau)\mathbb{R}_{A_i}-\tau \mathbb{R}_O$ and $\hat{f}(\beta)=\bigwedge_{i=1}^k\hat{h}_i(\beta)$. It is straight-forward to verify that $\arg\min \hat{h}_i(\beta)$ is achieved at its inflexion point $\beta^i=\left(\mathbb{G}^i_++\gamma\mathbb{G}^i_\Delta\right)^{-1}\left(\mathbb{Z}^i_++\gamma\mathbb{Z}^i_\Delta\right)$. With our new notation, the argmin of $f$ must be achieved at a point in either $\hat{B}_{inf}(\textbf{n})$ or $\hat{B}_{int}(\textbf{n})$, i.e.
		\begin{align}
			&\beta_\gamma=\arg\min_{\beta\in B_{inf} \cup  B_{int}\cap\R\cap\{\beta:\hspace{1mm} \exists 1\le i \le k,\hspace{1mm} f(\beta)=h_i(\beta)\}}f(\beta),
		\end{align}
		while
		\begin{align}\label{betaeqemp}
			&\hat{\beta}_\gamma(\textbf{n})=\arg\min_{\beta\in \hat{B}_{inf}(\textbf{n}) \cup \hat{B}_{int}(\textbf{n})\cap\R\cap\{\beta:\hspace{1mm} \exists 1\le i \le k, \hat{f}(\beta)= \hat{h}_i(\beta)\}} \hat{f}(\beta) .
		\end{align}
		We now construct the vector $V$ as follows, the first $k$ entries are $\beta_{inf(1)},...,\beta_{inf(k)}$ in chronological order. Then we place all the elements of $B^{i,j}$'s in rising order beginning with $B^{0,1}$ and ending with $B^{k-1,k}$ with the convention that we always place elements of $B^{i,j}$ for $i<j$ but not $i\ge j$. The individual elements of $B^{i,j}$ are to be ordered lexicographically. This way we can associate each argmin- candidate point with a unique index number in $V$. The number of elements in $V$ will be denoted $M$. We then construct $\hat{V}_n$ completely analogously from the corresponding estimators and let $\hat{M}^{i,j}(\textbf{n})$ be the number of elements in $\hat{V}_{\textbf{n}}$. We will later show that for large $\textbf{n}$, $\hat{M}^{i,j}(\textbf{n})=M$ a.s.. Let us define the optimal choice indices
		$$L=\arg\min_{1\le l \le M}f(V(l)), $$
		(note that $L$ is a set in general) so that $\beta_\gamma= g(V(L))$
		and similarly let 
		$$\hat{L}_{\textbf{n}}=\arg\min_{1\le l \le \hat{\mathcal{M}}(\textbf{n})}\hat{f}(\hat{V}_{\textbf{n}}(l), $$
		so that $\hat{\beta}_\gamma(\textbf{n})= \hat{f}(\hat{V}_{\textbf{n}}(l)), \forall l\in \hat{L}_{\textbf{n}}.$
		It may be the case that for some $\textbf{n}$, $\hat{L}_{\textbf{n}}\not\in L$. This can only happen in two different ways. The first one is that we do not have enough samples and the estimators deviate so much from their respective estimands that we make the wrong choice. The second one is due to $\mathbb{R}_i$ and $\mathbb{R}_j$ (for some $i$ and $j$) intersecting and inducing an argmin while $R_i$ and $R_j$ do not intersect. If $R_i$ and $R_j$ do not intersect they are separated by some fixed amount, for large enough $\textbf{n}$ so will $\mathbb{R}_i$ and $\mathbb{R}_j$ be, therefore this will not happen for large $\textbf{n}$. For any $l\not\in L$ there exists $a>0$ such that $V(l)-V(u)>a$, for any $u\in L$. We will show later that $\hat{V}_{\textbf{n}}(l)\xrightarrow{a.s.}V(l)$ for $1\le l\le M$ (recall that $\hat{M}^{i,j}_n=M$ for large $\textbf{n}$) and this will imply that $\hat{V}_{\textbf{n}}(l)-\hat{V}_{\textbf{n}}(u)>a$ for large enough $\textbf{n}$, $u\in L$. So for such large $\textbf{n}$ we must have $\hat{L}_{\textbf{n}}\in L$. So for large $\textbf{n}$ the right candidate is always picked and the theorem will follow if we can show that $\hat{V}_{\textbf{n}}(l)\xrightarrow{a.s.}V(l)$ for all $1\le l\le M$. This is clear by Theorem \ref{converge} for the inflexion points so it now remains to study possible minima along the intersections, i.e. we must show that with a lexicographical ordering of $B^{i,j}$, $\beta^{i,j}_1<...<\beta^{i,j}_q$ and of $\hat{B}^{i,j}(\textbf{n})$, $\hat{\beta}_1^{i,j}<...<\hat{\beta}_{\hat{q}}^{i,j}$ (with $\hat{q}=q$ for large $\textbf{n}$) we have that $\hat{\beta}^{i,j}_u\xrightarrow{a.s.}\beta^{i,j}_u$ for $u=1,...q$. 
		
		~~~\\
		\noindent
		\textit{Step 2: Use Lagrange multipliers to find possible minima along intersections of the risk functions.} 
		~~\\
		We now consider the corresponding empirical case, i.e. we assume that the two estimators $\hat{R}_i$ and $\hat{R}_j$ intersect. By continuity and the law of large numbers it follows that for large $\textbf{n}$, the matrix $\hat{G}_i$ is also positive definite and we may consider the corresponding diagonalization of $\hat{G}_i=\hat{O}\hat{D}\hat{O}^T$. If we let $\hat{\tilde{\lambda}}$ denote the smallest eigenvalue of $\hat{D}$. We can then show analogously to the population case that
		$$\hat{R}_{A_i}(\beta)\ge \frac{1}{n} \sum_{l=1}^{n_i}(Y^{A_i}_{l})^2+\n\beta\n_2\left(\hat{\tilde{\lambda}}\n\beta\n_2-2\sqrt{\frac{1}{n} \sum_{l=1}^{n_i}|Y^{A_i}_{l} |^2}\sqrt{\frac{1}{n} \sum_{l=1}^{n_i}\n X^{A_i}_{l}\n_2^2}\right).$$
		Therefore we may draw the same conclusion as before, namely if $\hat{R}_{A_i}$ and $\hat{R}_{A_j}$ intersect for large $\textbf{n}$ then we have an $\arg\min$ along this intersection. 
		Let $g(\beta)=R_{A_i}(\beta)-R_{A_j}(\beta)$. By the necessity part of the Kuhn-Tucker Theorem, a necessary condition for $\beta^*$ to be a minimum point for $R_{A_i}(\beta)$ subject to $g(\beta)=0$ is that $\nabla_\beta \mathcal{L}(\beta^*,\lambda^*)=0$ for some $\lambda^*$ and $g(\beta^*)=0$. This will only have a finite set of solutions in terms of $\lambda^*$ (as it will be polynomial in $\lambda^*$) and correspondingly a finite set of $\beta$'s, we then choose whichever one that minimizes $R_{A_i}$. Consider the Lagrangian
		\begin{align*}
			\mathcal{L}(\beta,\lambda)&=R_{A_i}(\beta)-\lambda g(\beta)
			=
			\sum_{l=0}^p\beta_l^2\left(\E\left[(X(l)^{A_i})^2-\lambda\left( (X(l)^{A_i})^2-(X(l)^{A_j})^2\right)\right] \right)
			\\
			&-2\sum_{l=1}^p\beta_l\E\left[Y^{A_i}(X(l)^{A_i})-\lambda\left(Y^{A_i}(X(l)^{A_i})-Y^{A_j}(X(l)^{A_j})\right)\right]
			\\
			&+2\sum_{l_1=0}^p\sum_{l_2\not=l_1}^p\beta_{l_1}\beta_{l_2}\E\left[X(l_1)^{A_i}X(l_2)^{A_i}-\lambda\left(X(l_1)^{A_i}X(l_2)^{A_i}-X(l_1)^{A_j}X(l_2)^{A_j}\right)\right]
			+
			\E\left[(Y^{A_i})^2-\lambda\left(Y^{A_i})^2-(Y^{A_j})^2\right)\right]
			\\
			&=\sum_{l=1}^p\beta_l^2\left(a_i(l)-\lambda\left( a_i(l)-a_j(l)\right) \right)
			-2\sum_{l=1}^p\beta_l\left(b_i(l)-\lambda\left(b_i(l)-b_j(l)\right)\right)
			\\
			&+2\sum_{l_1=0}^p\sum_{l_2\not=l_1}^p\beta_{l_1}\beta_{l_2}\left(a_i(l_1,l_2)-\lambda\left(a_i(l_1,l_2)-a_j(l_1,l_2)\right)\right)
			+c_i-\lambda\left(c_i-c_j\right).
		\end{align*}
		where $a_i(l_1,l_2)=\E\left[X(l_1)^{A_i}X(l_2)^{A_i}\right]$.
		Setting $\nabla_\beta\mathcal{L}(\beta,\lambda)=0$ yields,
		\begin{align}\label{LagrangeEq}
			&\beta_l\left(\E\left[(X(l)^{A_i})^2-\lambda\left( (X(l)^{A_i})^2-(X(l)^{A_j})^2\right)\right] \right)
			=
			\E\left[Y^{A_i}(X(l)^{A_i})-\lambda\left(Y^{A_i}(X(l)^{A_i})-Y^{A_j}(X(l)^{A_j})\right)\right]
			\nonumber
			\\
			&-\sum_{l_2\not=l}^p\beta_{l_2}\E\left[X(l)^{A_i}X(l_2)^{A_i}-\lambda\left(X(l)^{A_i}X(l_2)^{A_i}-X(l)^{A_j}X(l_2)^{A_j}\right)\right].
		\end{align}
		We define the $p\times p$ matrix $M^{i,j}(\lambda)$ and the vector $C^{i,j}(\lambda)$ exactly as in \ref{parammatrix} using the given environments i.e.
		$$M^{i,j}(\lambda)_{l,l} =a_i(l)-\lambda\left(a_i(l)-a_j(l)\right)$$
		when $u\not=v$
		$$M^{i,j}(\lambda)_{u,v} =a_i(u,v)-\lambda\left(a_i(u,v)-a_j(u,v)\right)$$
		and
		$$C^{i,j}(\lambda)_u=b_i(u)-\lambda\left(b_i(u)-b_j(v)\right).$$
		Define
		$$d_l(\lambda)=\E\left[(X(l)^{A_i})^2-\lambda\left( (X(l)^{A_i})^2-(X(l)^{A_j})^2\right)\right],$$ 
		$$e_l(\lambda)=\sum_{l_2\not=l}^p\beta_{l_2}\E\left[X(l)^{A_i}X(l_2)^{A_i}-\lambda\left(X(l)^{A_i}X(l_2)^{A_i}-X(l)^{A_j}X(l_2)^{A_j}\right)\right]\hspace{2mm} and$$
		$$c_l(\lambda)=\E\left[Y^{A_i}(X(l)^{A_i})-\lambda\left(Y^{A_i}(X(l)^{A_i})-Y^{A_j}(X(l)^{A_j})\right)\right].$$
		With this notation we can reformulate \eqref{LagrangeEq} as
		$$M^{i,j}(\lambda)\beta(\lambda)=C^{i,j}_{\lambda},$$ 
		where we write $\beta(\lambda)$ only to signify that $\beta$ depends on $\lambda$.
		\\
		\\ 
		\textit{Step 3: Only finitely many Lagrange multiplier candidates with corresponding solutions $\beta(\lambda)$ (outside a set of measure zero).}
		\\
		Let $P(\lambda)=\det\left(M^{i,j}(\lambda)\right)$. We will now establish the fact that outside a set of measure zero in $\Theta\times\Theta$, $M(\lambda_u)\beta=C^{i,j}(\lambda)$ has no solutions for any $\lambda\in\R$ such that $P(\lambda)=0$. By Lemma \ref{matrixLemma} outside a set of measure zero $N_2\in\Theta\times\Theta$, $\mathrm{rank}(M^{i,j}(\lambda))=p-1$ for $\lambda\in\Lambda$ and 
		\begin{align}\label{linearcomb}
			&M^{i,j}_{p,.}(\lambda)=\sum_{u=1}^{p-1}s_u(\theta,\lambda)M_{u,.}(\lambda)
		\end{align}
		with all $s_u(\theta,\lambda)\not=0$. 
		From Lemma \ref{matrixNoTiesLemma} we see that outside a zero set $N_3\in\Theta\times\Theta$, $\{s_u(\theta,\lambda)\}_{u=1}^{p-1}$ are all rational functions on $\R\times\Theta\times\Theta$. Applying the same row operations to the vector $C^{i,j}(\lambda)$ as $M^{i,j}(\lambda)$ we now have that $M^{i,j}(\lambda)\beta=C^{i,j}(\lambda)$ has no solutions if
		$$\left(C^{i,j}_{\lambda}(p)-\sum_{v\not=p}s_v(\theta,\lambda) C^{i,j}_{\lambda}(v)\right)\not=0.$$
		Let $G$ denote the lowest common denominator of the terms in $\left(C^{i,j}_{\lambda}(p)-\sum_{v\not=p}s_v(\theta,\lambda) C^{i,j}_{\lambda}(v)\right)$ then 
		$$Q(\theta,\lambda)= G(\theta,\lambda)\left(C^{i,j}_{\lambda}(p)-\sum_{v\not=p}s_v(\theta,\lambda) C^{i,j}_{\lambda}(v)\right)$$ 
		is a polynomial on $\R\times\Theta\times\Theta$. To show that $Q$ has no roots in $\Lambda$ is equivalent to showing that $P$ and $Q$ have no common roots which is in turn equivalent to showing that their resultant is not zero. Again, we need to show that there exists $\theta\in\Theta$ such that $P$ and $Q$ have no common roots. Take the matrix $\tilde{M}$ from the proof of Lemma \ref{matrixNoTiesLemma}. As we already have seen for $\tilde{M}$, $s_1\equiv...\equiv s_{p-1}\equiv 1$. We readily see that in this case $G\equiv 1$ (since $G$ is the lowest common denominator of $\{s_u\}_{u=1}^{p-1}$) so any choice of parameters that leads to the first degree polynomial $C^{i,j}_{\lambda}(p)-\sum_{v\not=p}s_v(\theta,\lambda) C^{i,j}_{\lambda}(v)$ not having a root in $\lambda=0$ shows that outside a zero set $N_4$ there are no solutions to $M^{i,j}(\lambda_u)\beta=C^{i,j}(\lambda)$ for any $\lambda\in\Lambda$. 
		\\
		\\
		\textit{Step 4: For all viable candidates $\beta(\lambda)$, find the roots of $g(\beta(\lambda))=0$ and show there are only finitely many candidates.}
		\\
		We continue with the case when $\lambda\in\R$ is such that $M^{i,j}(\lambda)$ is full rank. If $M^{i,j}(\lambda)$ is full-rank then $\beta(\lambda)=(M^{i,j})^{-1}(\lambda)C^{i,j}_{\lambda}$, to find $\lambda$ we solve $g(\beta(\lambda))=0$. Note that 
		$$\beta(\lambda)_u=\frac{1}{\det\left(M^{i,j}(\lambda)\right)}\sum_{v=1}^p\det\left(M^{i,j}_{u,v}(\lambda)\right)C^{i,j}(\lambda)(v),$$ 
		where we again use the convention that $M^{i,j}_{u,v}$ is the matrix $M^{i,j}$ with row $u$ and column $v$ removed. So for $\lambda\not\in\Lambda$
		\small
		\begin{align*}
			&g(\beta(\lambda))=\left(a_i(l)-a_j(l)\right)\sum_{l=1}^p\left(\frac{1}{\det\left(M^{i,j}(\lambda)\right)}\sum_{v=1}^p\det\left(M^{i,j}_{l,v}(\lambda)\right)(-1)^{l+v}C^{i,j}(\lambda)(v)\right)^2+c_i-c_j-
			\\
			&2\left(b_i(l)-b_j(l)\right)\sum_{l=1}^p\left(\frac{1}{\det\left(M^{i,j}(\lambda)\right)}\sum_{v=1}^p\det\left(M^{i,j}_{l,v}(\lambda)\right)(-1)^{l+v}C^{i,j}(\lambda)(v)\right)+
			\\
			&2\left(a_i(l_1,l_2)-a_j(l_1,l_2)\right)\sum_{l_1=1}^p\sum_{l_2\not=l_2}\left(\frac{1}{\det\left(M^{i,j}(\lambda)\right)}\sum_{v=1}^p\det\left(M^{i,j}_{l_1,v}(\lambda)\right)(-1)^{l_1+v}C^{i,j}(\lambda)(v)\right)\cdot
			\\
			&\left(\frac{1}{\det\left(M^{i,j}(\lambda)\right)}\sum_{v=1}^p\det\left(M_{l_2,v}(\lambda)\right)(-1)^{l_2+v}C^{i,j}(\lambda)(v)\right).
		\end{align*}
		\normalsize
		Furthermore from the above expression we see that for $\lambda\not\in\Lambda$, $g(\beta(\lambda))$ has the same roots as 
		$$\tilde{P}(\lambda)=\det\left(M^{i,j}(\lambda)\right)^2g(\beta(\lambda)),$$ 
		which is a polynomial. Outside a set of measure zero $N_5\in\Theta\times\Theta$, $\tilde{P}(\lambda)$ has only simple roots, according to Lemma \ref{discriminantLemma}. Let us now define $N=N_1\cup...\cup N_5$. Letting $\mathcal{R}$ denote the (finite) set of roots to $g(\beta(\lambda))=0$ then the set of intersection $\arg\min$'s will now be given by 
		$$B^{i,j}=\left\{M^{-1}(\lambda^{i,j})C_{\lambda^{i,j}}: \lambda^{i,j}=\arg\min_{\lambda\in\mathcal{R}}R_{A_i}\left(M^{-1}(\lambda)C^{i,j}(\lambda)\right)\right\}.$$
		\\
		\textit{Step 5: Plug-in estimator for candidate points along intersections converge to corresponding points in population case.}
		\\
		From here on out we fix $\theta\in\Theta\times\Theta\setminus N$. We now consider the plug-in estimator of the minima along the intersection between $R_{A_i}$ and $R_{A_j}$,
		$$\arg\min_{\beta\in\R^p} \hat{R}_{A_i}(\beta),\texttt{ subject to } \hat{R}_{A_i}(\beta)=\hat{R}_{A_j}(\beta).$$
		We also consider the plug-in Lagrangian
		\begin{align*}
			&\left(\hat{\mathcal{L}}(\beta,\lambda)\right)(\textbf{n})=\left(\hat{R}_{A_i}(\beta)\right)(\textbf{n})-\lambda \hat{g}(\beta)
			=
			\sum_{l=1}^p\beta_l^2\left(-\lambda\left( \hat{a}_i(l)-\hat{a}_j(l)\right) \right)
			-2\sum_{l=1}^p\beta_l\left(\hat{b}_i(l)-\lambda\left(\hat{b}_i(l)-\hat{b}_j(l)\right)\right)
			\\
			&2\sum_{l_1=0}^p\sum_{l_2\not=l_1}^p\beta_{l_1}\beta_{l_2}\left(\hat{a}_i(l_1,l_2)-\lambda\left(\hat{a}_i(l_1,l_2)-\hat{a}_j(l_1,l_2)\right)\right)
			+\hat{c}_i-\lambda\left(\hat{c}_i-\hat{c}_j\right),
		\end{align*}
		where $\hat{g}(\beta)=\hat{R}_{A_i}(\beta)-\hat{R}_{A_j}(\beta)$, $\hat{a}_i(l)=\frac{1}{n_{A_i}}\sum_{u=1}^{n_{A_i}}\left(X^{A_i}_u(l)\right)^2$, $\hat{b}_i(l)=\frac{1}{n_{A_i}}\sum_{u=1}^{n_{A_i}}X^{A_i}_u(l)Y^{A_i}$, $\hat{c}_i=\frac{1}{n_{A_i}}\sum_{u=1}^{n_{A_i}}\left(Y^{A_i}\right)^2$, $\hat{a}_i(l_1,l_2)=\frac{1}{n_{A_i}}\sum_{u=1}^{n_{A_i}}X^{A_i}_u(l_1)X^{A_i}(l_2)$. Then similarly we see that solving $\nabla_\beta \left(\hat{\mathcal{L}}(\beta,\lambda)\right)(\textbf{n})=0$ leads to $\widehat{M}^{i,j}(\lambda)\beta=\widehat{C}^{i,j}_\lambda$, where 
		$$\widehat{C}^{i,j}_u(\lambda)=\hat{b}_i(u)-\lambda\left(\hat{b}_i(u)-\hat{b}_j(v)\right), $$
		$$\widehat{M}^{i,j}_{l,l}(\lambda)=\hat{a}_i(l)-\lambda\left(\hat{a}_i(l)-\hat{a}_j(l)\right),$$
		and when $u\not=v$
		$$\widehat{M}^{i,j}(\lambda)_{u,v} =\hat{a}_i(u,v)-\lambda\left(\hat{a}_i(u,v)-\hat{a}_j(u,v)\right).$$
		Let $\hat{\Lambda}_n$ denote the set of roots (in terms of $\lambda)$ to $\hat{P}(\lambda)= \det\left(\hat{M}^{i,j}(\lambda)\right)=0$ and let $\hat{\Lambda}_\R(\textbf{n})$ denote the real roots of this polynomial. We wish to establish that for large enough $\textbf{n}$, $\hat{M}^{i,j}(\lambda)=\hat{C}^{i,j}_\lambda$ has no solutions for $\lambda\in\hat{\Lambda}_n$. Recall that $P$ only has simple roots which implies, by Lemma \ref{simpleroots} and the law of large numbers that the elements in $\hat{\Lambda}_\R(\textbf{n})$ converge a.s. to those in $\Lambda_\R$, where we let $\Lambda_\R=\{\lambda_1,...,\lambda_m\}$ (where $\lambda_1<...<\lambda_m$ and $m$ depends on $\theta$ but $\theta$ is fixed at this point) denote the real roots of $P$. Note that by the proof of \ref{matrixLemma}, for each $\lambda\in\Lambda$ (and therefore also in $\Lambda_\R$) the first $p-1$ first minors along the first column in $M^{i,j}(\lambda)$ is bounded below by some $\delta>0$ ($\delta>0$ also depends on $\theta$, but $\theta$ is fixed), i.e. $\left|\det\left(M_{u,1}(\lambda)\right)\right|>\delta$ for each $\lambda\in\Lambda$ and $1\le u\le p-1$. For large enough $\textbf{n}$, the number of elements in $\hat{\Lambda}_n$ and $\Lambda$ coincides by Lemma \ref{simpleroots}. From Lemma \ref{simpleroots} we also have $\hat{\lambda}_z\xrightarrow{a.s.}\lambda_z$, for $1\le z\le m$. Since the determinant is continuous it follows that $\left|\det\left(\hat{M}^{i,j}_{1,v}(\hat{\lambda}_z)\right)\right|>\delta$, $v=1,...,p-1$ for large enough $\textbf{n}$. Therefore $\mathsf{rank}\left( \hat{M}^{i,j}(\hat{\lambda}_z)\right)=p-1$ and we have the unique representation (in terms of the rows $\hat{M}^{i,j}_{u,.}(\lambda)$)
		\begin{align}\label{linearcombhat}
			&\hat{M}^{i,j}_{p,.}(\hat{\lambda}_z)=\sum_{u=1}^{p-1}\hat{s}_u(\hat{\lambda}_z)\hat{M}^{i,j}_{u,.}(\hat{\lambda}_z)
		\end{align}
		for all $\hat{\lambda}_z\in\hat{\Lambda}(\textbf{n})_\R$ with all $\hat{s}_u(\hat{\lambda}_z)\not=0$. Note that 
		\begin{align}\label{Mdiff}
			&\n \hat{M}^{i,j}(\hat{\lambda}_z)-M(\lambda_z) \n\le \n \hat{M}^{i,j}(\hat{\lambda}_z-\lambda_z) \n+\n\hat{M}^{i,j}(\lambda_z)-M(\lambda_z) \n,
		\end{align}
		for the first term above $\n \hat{M}^{i,j}(\hat{\lambda}_z-\lambda_z) \n=\left|\hat{\lambda}_z-\lambda_z\right|\n\hat{M}'(\textbf{n})\n$, where $\hat{M}'_{u,v}(\textbf{n})=\hat{w}_i(u,v)-\hat{w}_j(u,v)$ when $u\not=v$ and $\hat{M}'_{l,l}(\textbf{n})=\hat{a}_i(l)-\hat{a}_j(l)$. By the law of large numbers $\hat{M}'(\textbf{n})$ converges to the matrix $M'$ with $M'_{u,v}=a_i(u,v)-a_j(u,v)$ when $u\not=v$ and $M'_{l,l}=a_i(l)-a_j(l)$. Due to the fact that $\hat{\lambda}_z\xrightarrow{a.s.}\lambda_z$ it is therefore clear that the first term in \eqref{Mdiff} vanishes. The second therm in \eqref{Mdiff} will vanish by the law of large numbers, hence $\n \hat{M}^{i,j}(\hat{\lambda}_z)-M^{i,j}(\lambda_z) \n\xrightarrow{a.s.}0$.
		Due to the fact that $\n \hat{M}^{i,j}(\hat{\lambda}_z)-M^{i,j}(\lambda_z) \n\xrightarrow{a.s.}0$ and Lemma \ref{coefficientslemma} it follows that $\hat{s}_u(\hat{\lambda}_z)\xrightarrow{a.s.}s_u(\lambda_z)$ for $u=1,...,p-1$ and $1\le z\le m$. Let
		$$H(\lambda)= \frac{Q(\lambda)}{G(\lambda)} = \left(C^{i,j}_{p}(\lambda)-\sum_{v\not=p}s_v(\theta,\lambda) C^{i,j}_{v}(\lambda)\right),$$
		we know that since we chose $\theta\not\in N$, $|H(\lambda)|>\delta'$ (since $Q(\lambda)\not=0$) for some $\delta'>0$ and every $\lambda\in \Lambda(\theta)$. We now let 
		$$\hat{H}(\lambda)= \left(\hat{C}^{i,j}_{p}(\lambda)-\sum_{v\not=p}\hat{s}_v(\lambda) \hat{C}^{i,j}_{v}(\lambda)\right),$$ 
		it follows immediately that $\hat{H}(\hat{\lambda}_z)\xrightarrow{a.s.}H(\lambda_z)$ by the fact that $\hat{s}_u(\hat{\lambda}_z)\xrightarrow{a.s.}s_u(\lambda_z)$ and $\hat{C}^{i,j}_{\hat{\lambda}_z}(v)\xrightarrow{a.s.}C^{i,j}_{\lambda_z}(v)$ for $v=1,...p$ (this latter statement can be proved by employing the same strategy as for \eqref{Mdiff}). So for large $\textbf{n}$, $\hat{H}(\hat{\lambda}_z)>\delta$ which implies that there are no solutions to $\hat{M}^{i,j}(\lambda)\beta=\hat{C}^{i,j}(\lambda)$ for $\lambda\in \hat{\Lambda}_\R(\textbf{n})$ for large enough $\textbf{n}$. Let $\hat{\beta}(\lambda)=(\hat{M}^{i,j})^{-1}(\lambda)\hat{C}^{i,j}(\lambda)$ for $\lambda\not\in \hat{\Lambda}_n$. To find any candidates for $\hat{\lambda}$ we then solve $\hat{g}(\beta(\lambda))=0$. Similar to before we see that
		\small
		\begin{align}\label{groots}
			\hat{g}(\beta(\lambda))&=\sum_{l=1}^p\left(\frac{1}{\det\left(\hat{M}^{i,j}(\lambda)\right)}\sum_{v=1}^p\det\left(\hat{M}^{i,j}_{l,v}(\lambda)\right)(-1)^{l+v}\hat{C}^{i,j}_\lambda(v)\right)^2\left(\hat{a}_i(l)-\hat{a}_j(l)\right)+\hat{c}_i-\hat{c}_j\nonumber
			\\
			&-2\sum_{l=1}^p\left(\frac{1}{\det\left(\hat{M}^{i,j}(\lambda)\right)}\sum_{v=1}^p\det\left(\hat{M}^{i,j}_{l,v}(\lambda)\right)(-1)^{l+v}\hat{C}^{i,j}_\lambda(v)\right)\left(\hat{b}_i(l)-\hat{b}_j(l)\right)\nonumber
			\\
			&+2\left(\hat{a}_i(l_1,l_2)-\hat{a}_j(l_1,l_2)\right)\sum_{l_1=1}^p\sum_{l_2\not=l_2}\left(\frac{1}{\det\left(\hat{M}^{i,j}(\lambda)\right)}\sum_{v=1}^p\det\left(\hat{M}^{i,j}_{l_1,v}(\lambda)\right)(-1)^{l_1+v}\hat{C}^{i,j}_\lambda(v)\right)
			\nonumber
			\\
			&\times\left(\frac{1}{\det\left(\hat{M}^{i,j}(\lambda)\right)}\sum_{v=1}^p\det\left(\hat{M}^{i,j}_{l_2,v}(\lambda)\right)(-1)^{l_2+v}\hat{C}^{i,j}_\lambda(v)\right).
		\end{align}
		\normalsize
		Analogously to the population case, if let $\hat{\tilde{P}}(\lambda)=\det\left(\hat{M}^{i,j}(\lambda)\right)^2\hat{g}(\beta(\lambda))$ then $\hat{\tilde{P}}(\lambda)$ has the same roots as $\hat{g}(\beta(\lambda))$. By the law of large numbers and continuity we see that $\det\left(\hat{M}^{i,j}(\lambda)\right)\xrightarrow{a.s.}\det\left( M^{i,j}(\lambda)\right)\not=0$ (since $\lambda\not\in\Lambda$), $\det\left(\hat{M}^{i,j}_{u,v}(\lambda)\right)\xrightarrow{a.s.} \det\left(M^{i,j}_{u,v}(\lambda)\right)$ for all $u,v$, $\hat{b}_.(.)\xrightarrow{a.s.}b_.(.)$, $\hat{a}_.(.)\xrightarrow{a.s.}a_.(.)$, $\hat{a}_.(.,.)\xrightarrow{a.s.}a_.(.,.)$ and finally that $\hat{c}_.\xrightarrow{a.s.}c_.$. We can thus conclude that all coefficients of $\hat{\tilde{P}}$ converge to the corresponding coefficients of $\tilde{P}$ and therefore, by Lemma \ref{simpleroots}, the real roots of $\hat{\tilde{P}}$ converge to those of $\tilde{P}$ and that for large enough $\textbf{n}$, the number of such roots of $\tilde{P}$ and $\hat{\tilde{P}}$ coincides. Letting $\mathcal{R}=\{\lambda_1,...,\lambda_m\}$ (with $\lambda_1<...<\lambda_{\hat{m}}$) and $\hat{\mathcal{R}}=\{\hat{\lambda}_1,...,\hat{\lambda}_{\hat{m}}\}$ (with $\hat{\lambda}_1<...<\hat{\lambda}_{\hat{m}}$) denote the simple real roots of $\tilde{P}$ and $\hat{\tilde{P}}$ respectively we have that $\hat{\lambda}_u\xrightarrow{a.s.}\lambda_u$ for $u=1,..,m$ and for large $n$, $\hat{m}=m$, i.e. the elements of $\hat{\mathcal{R}}$ converge to those of $\mathcal{R}$. If we now define
		$$\hat{B}^{i,j}(\textbf{n})=\left\{\hat{M}^{-1}(\lambda^{i,j})\hat{C}^{i,j}(\lambda^{i,j}): \lambda^{i,j}=\arg\min_{\lambda\in\hat{\mathcal{R}} }\hat{R}_{A_i}\left((\hat{M}^{i,j})^{-1}(\lambda)\hat{C}^{i,j}_\lambda\right)\right\}. $$
		By the law of large numbers and continuity (since $\det\left(M^{i,j}(\lambda)\right)\not=0$ for $\lambda\in\Lambda$) $(\hat{M}^{i,j})^{-1}(\lambda)\hat{C}^{i,j}_{\lambda}\xrightarrow{a.s.}(M^{i,j})^{-1}(\lambda)C^{i,j}(\lambda)$ for $\lambda\not\in\Lambda$. Therefore $\max_{\beta\in\hat{B}^{i,j}}\dist\left(\beta,B^{i,j}\right)\xrightarrow{a.s.}0$ as well as $\#\hat{B}^{i,j}(\textbf{n})=\#B^{i,j}$ (with $\#A$ denoting the number of elements in the set $A$) for large $\textbf{n}$.
		With a lexicographical ordering of $B^{i,j}$, $\beta^{i,j}_1<...<\beta^{i,j}_q$ and of $\hat{B}^{i,j}(\textbf{n})$, $\hat{\beta}_1^{i,j}<...<\hat{\beta}_{\hat{q}}^{i,j}$ (with $\hat{q}=q$ for large $\textbf{n}$) we have that $\hat{\beta}^{i,j}_u\xrightarrow{a.s.}\beta^{i,j}_u$ for $u=1,...q$.
	\end{proof}



\end{document}